\documentclass{birkjour}
\usepackage{amsmath, amsthm, amssymb, fancyhdr}
\usepackage{amsfonts}
\usepackage{ newtxtext,newtxmath}
\usepackage[ansinew]{inputenc}
\usepackage[dvips]{epsfig}
\usepackage{graphicx, tikz}
\usepackage[percent]{overpic}
\usepackage[english]{babel}
\usepackage{cite}
\usepackage{graphicx}
\usepackage{amscd}
\usepackage{xcolor}
\usepackage{bm}
\usepackage{enumerate}

\usepackage{verbatim}
\usepackage{hyperref}
\usepackage{amstext}
\usepackage{latexsym}
\let\oldsqrt\sqrt
\def\sqrt{\mathpalette\DHLhksqrt}
\def\DHLhksqrt#1#2{%
\setbox0=\hbox{$#1\oldsqrt{#2\,}$}\dimen0=\ht0
\advance\dimen0-0.2\ht0
\setbox2=\hbox{\vrule height\ht0 depth -\dimen0}%
{\box0\lower0.4pt\box2}}

\allowdisplaybreaks

\newcommand{\loglap}{L_{\text{\tiny $\Delta \:$}}\!}
\newcommand{\logrel}{ \text{$(I-\Delta)^{\log}$}}
\newcommand{\logrels}{ \text{$(I+(-\Delta)^s)^{\log}$}}

\newcommand{\R}{\mathbb{R}} % reelle Zahlen
\newcommand{\N}{\mathbb{N}} % natuerliche Zahlen
\newcommand{\dist}{\textnormal{dist}} % dist ...
\newcommand{\supp}{\textnormal{supp}} % supp ...
\renewcommand{\H}{{\mathcal H}}

\renewcommand{\phi}{\varphi}

\newcommand{\cC}{{\mathcal C}}

\newcommand{\cF}{{\mathcal F}}

\newcommand{\cH}{{\mathcal H}}

\newcommand{\cL}{{\mathcal L}}

\newcommand{\cS}{{\mathcal S}}

\newtheorem{defi}{Definition}[section]
\newtheorem{remark}[defi]{Remark}

\theoremstyle{plain} %default%plain
\newtheorem{thm}[defi]{Theorem}
\newtheorem{prop}[defi]{Proposition}
\newtheorem{lemma}[defi]{Lemma}

\numberwithin{equation}{section}

\title
 [The fractional logarithmic  Schr\"odinger operator]{\center The fractional logarithmic  Schr\"odinger operator:  properties and functional spaces}

\thanks{This work was initiated when the author visited Rutgers University of New-Jesey, USA, through  the Abbas Bahri Excellence Fellowship.  He is grateful to the Department of Mathematics for the hospitality and wishes to thank   Yan Yan Li,  Zheng-chao Han, and  Denis Kriventsov for their helpful discussions. The author thanks S. Jarohs for helpfull comments and R. Frank for pointing out paper \cite{FKT20}. The author is  partially funded by a Research Strategic Initiatives Grant (RSIG22) from the University of Northern British Columbia, Canada.}

\author[P. A. Feulefack]{\centering Pierre Aime Feulefack }
\address{Department of Mathematics and Statistics, University of Northern British Columbia, 		Prince George, BC, Canada}
\email{pierreaimefeulefack@unbc.ca}
\begin{document}
\maketitle

\begin{abstract}
In this note, we deal with  the  fractional logarithmic Schr\"odinger operator $(I+(-\Delta)^s)^{\log}$ and the corresponding energy spaces for variational study.  The  fractional (relativistic) logarithmic Schr\"odinger operator is the  pseudo-differential operator with logarithmic Fourier symbol, $\log(1+|\xi|^{2s})$, $s>0$. We first   establish the integral representation corresponding to the operator and  provide an asymptotics property of the related kernel. We  introduce    the functional analytic theory  allowing to study  the  operator from a PDE point of view and the associated  Dirichlet problems in an open set of $\mathbb{ R}^N.$ We also establish some variational  inequalities,     provide the fundamental solution and  the asymptotics  of the corresponding Green function at zero and at infinity.

\end{abstract}

{\footnotesize
\begin{center}
\textit{Keywords.}   Logarithmic symbol,     pseudo-differential operator,    Schr\"odinger operator, Green function, fundamental solution, regularity,  maximum principle.
\end{center}
\vspace{.4cm}
\textit{2010 Mathematics Subject Classification}. 45P05, 45C05,  45A05, 	35R11, 35S15.
}

\tableofcontents

\section{Introduction  }\label{intro}%%%%%%%%%%%%%%%%%%%%%%%%%

This note provides tools and properties for the fractional (relativistic) logarithmic Schr\"odinger operator denoted by $(I+(-\Delta)^s)^{\log}$,     which is the pseudo-differential operator corresponding to the  logarithmic Fourier symbol $$\log(1+|\xi|^{2s}),\quad s>0,$$ 
which, for smooth function $u:\R^N\to \R$, it is equivalent to the equality
$$(I+(-\Delta)^s)^{\log}u(x)= \cF^{-1}\big(\log(1+|\xi|^{2s})\cF(u)(\xi)\big)(x), \qquad (x\in \R^N).$$
Here and in the following, $\cF$ and $\cF^{-1}$ stand for the Fourier transform and inverse Fourier transform respectively. This note is somehow a continuation of the logarithmic  Schr\"odinger operator corresponding to the particular  case $s=1$, denoted by 
 $\logrel$,  introduced in \cite{F22},  
$$(I-\Delta)^{\log}u(x)= \cF^{-1}\big(\log(1+|\xi|^{2})\cF(u)(\xi)\big)(x), \qquad (x\in \R^N).$$
As already mentioned in \cite{F22}, we  adopt  the superscript  notation  $(I+(-\Delta)^s)^{\log}$  to emphasize on the nonlocal feature of these operators.  We    call  the operator here  the fractional   logarithmic  Schr\"odinger operator  and this can be found in literature with the classical notation 
$$\log(I+(-\Delta)^s),$$
see for instance \cite{CPI22,PCI22,SSV06,KM14,B14,GR12,BG60,SV03,B15,ZR05,KP99,B15,MCC98}  and the references therein. These operators are usually called  in the point of view of probabilistic and  potential theory,  the  infinitesimal generators of a symmetric geometric $2s$-stable processes for $s\in (0,1)$ \cite{JS01} and  the infinitesimal generator for the symmetric  variance gamma process for $s=1$ \cite{KS13}. Generally for $s>0$, the fractional  logarithmic Schr\"odinger operators are    defined through their characteristic function or their Laplace exponent $$\psi(\lambda):=\log(1+\lambda^{2s}).$$
 In that direction, the fractional logarithmic Schr\"odinger operator has so far received most attentions and has vaste applications  in mathematical finance. In particular, they play an important role in heavy-tail modeling of economic data \cite{MCC98} and in the study heavy-tailed financial models \cite{RM00, SSV06}. Moreover, despite the wide  applications  in mathematical finance and other fields of sciences, there has not been much study in the point of view of Partial Differential Equations (PDEs) or Integro-Differential Equations (IDEs) involving   these operators.
 
 Many results in the literature on problems involving   the fractional  Logarithmic Schr\"odinger operator are proved using the corresponding L\'evy Khintchine or  Laplace exponent  $\psi(\lambda):=\log(1+\lambda^{2s})$, combined with Tauberian-type theorems \cite{SSV06,KM14,KP99,KM17,M14} or the method via Fourier transform in $\R^N$. We mention for instance  \cite{PCI22,CPI22}   where the fractional  logarithmic Schr\"odinger operator have been  used in wave equation to model damping mechanism in $\R^N$ (see also the references therein)
 \begin{equation*}\label{inv}
	\begin{split}
	\quad\left\{\begin{aligned}
		u_{tt}+\logrels u+(I+(-\Delta)^s)^{\log}  u_t &= 0 &&\text{ \ \quad  $x\in \R^N$,\ $t>0$}\\
		u(x,0) =  u_0(x)  \quad u_t(x,0)= u_1(x) &         && \text{ \ \  } x\in \R^N. 
	\end{aligned}\right.
	\end{split}
	\end{equation*}
 We also mention  \cite{DCW2011}, where  the fractional  logarithmic Schr\"odinger operator has been  used in an inviscid model generalizing the two-dimensional Euler and the surface quasi-geostrophic equations describe by
 \begin{equation*}\label{dam}
	\begin{split}
	\quad\left\{\begin{aligned}
		&\theta_{t}+u\cdot \nabla \theta = 0\quad \text{ \   $x\in \R^N$,\ $t>0$}\\
		&u =  \nabla^{\perp} \psi, \quad \Delta \psi = \logrels \theta,           \text{ \ \  } 
	\end{aligned}\right.
	\end{split}
	\end{equation*}
 where $\theta=\theta(x, t)$ is a scalar function of $x\in\R^N$ and $t>0$, $u$ denotes a velocity field in $\R^N$. This is also known as log-Euler equation. In   \cite{ZSGZ23}, the logarithmic Schr\"odinger operator is used to model a fast algorithm for intra-frame versatile video coding based on edge features. In \cite{LP05} a model involving a one-dimensional logarithmic Schr\"odinger operator is used to investigate finite-time blow-up and stability of semilinear integro-differential equations of the form
 \[
 \begin{cases}
     w_t+\logrel w+\nu t^{\sigma}w^{1+\beta}\\
     w(x,0)=\phi(x)\ge 0,\quad x\in \R_+,
 \end{cases}
 \]
 where  $\nu>0$, $\sigma\in \R$ and $\beta>0$ are constants. See also \cite{CK23,ZK22, AR23} and the references in there for other problems involving the logarithmic Schr\"odinger operator $\logrel$.
 
 In the general theory of partial differential equations (PDEs) or integro-differential equations (IDEs), to avoid apriori regularity assumption on solutions to the initial problem, it is a common case that one first finds a weak or distributional solutions  to the initial problem, and then, by showing that the solution is sufficiently regular, one argues that it is in fact a strong or classical solution. This  theory of  PDEs involves Sobolev spaces and variational approach, which are basic tools to establish weak, distributional or viscosity solutions to a specific given equation.

The aim of the present note is to define   the operator $\logrels$ from an analytical  point of view and  introduce tools  from  functional analytic theory which, we think will be helpful  to PDEs oriented readers to  study   problems involving the fractional  logarithmic Schr\"odinger operator $\logrels$ in a domain using variational approach. So far
, there are not many papers or books in the literature dealing with the fractional logarithmic Schr\"odinger operator in domain. 

In the sense above, some of our results may not be new for expert in nonlocal elliptic or parabolic operators,   but the exposition is explicitly given in this note particularly for the fractional logarithmic schr\"odinger operator. Moreover, more on L\'evy-type processes and their corresponding  pseudo-differential operators on the general form  $\psi(\Delta)$, where $\psi$ is a Bernstein functions can be found in      \cite{JS01}. \\

\noindent
{\bf Outline.} The paper is divided in four  sections and   organised as follows. \\

In section \ref{sECT1}, we show that for  compactly supported Dini continuous function $u:\R^N\to\R,$ the operator $\logrels$ can be represented as a singular integral operator,
\[
\begin{split}
(I+(-\Delta)^s)^{\log} u(x)&=\cF^{-1}\big(\log(I+|\xi|^{2s})\cF(u)\big)(x)\\&=P.V\int_{\R^N}\big(u(x)-u(y)\big)K_s(x-y) \ dy,
\end{split}
\]
where the  kernel function $K_s$ is given below (see \eqref{KerenelExp}) and  satisfies the asymptotics
\begin{equation}\label{A}
K_s(z)\sim \begin{cases} |z|^{-N} \quad& \text{ as} \quad |z|\to 0\\
\\
|z|^{-N-2s} \quad &\text{ as} \quad |z|\to \infty
\end{cases}
\end{equation}
for $s\in (0,1)$, and for $s=1$,
\begin{equation}\label{A-1}
K_1(z)\sim \begin{cases} |z|^{-N} \quad& \text{ as} \quad |z|\to 0\quad \\
\\
|z|^{-\frac{N+1}{2}}e^{-|z|} \quad &\text{ as} \quad |z|\to \infty.
\end{cases}
\end{equation}

In section \ref{Sect2}, we deal with the fundamental solution and  Green function for $\logrels$. Denoting by $G$ the Green function for $\logrels$, we show that 
\[
G(x)=\int_0^{\infty}q_t(x)\ dt=\int_0^{\infty} \frac{1}{\Gamma(t)}\int_{0}^{\infty}p_s(x,\tau)\tau^{t-1}e^{-\tau}\ d\tau\ dt.
\]
 where $q_t$ is the associated density function and the function $G$ satisfies  the asymptotics
    \begin{align*}
G(x)\sim \begin{cases}
\frac{1}{|x|^{N}(\log\frac{1}{|x|})^2}&\qquad\text{as}\quad |x|\to 0\\
\\
 \frac{C_{N,s}}{|x|^{N+2s}(\log\frac{1}{|x|})^2} &\qquad\text{ as }\quad |x|\to \infty.
\end{cases}
\end{align*}

In section \ref{Sect3}, we set up the functional analytic framework for the  Dirichlet problems involving the operator $\logrels$ in domain. We prove  some important  inequalities for variational study as well as some compactness results.

In section \ref{Sect4}, we investigate the existence results   and some properties of solutions for Poisson problem involving the fractional logarithmic  Schr\"odinger $\logrels$ in an open bounded  set $\Omega$ of $\R^N$. Moreover, we prove also the strong maximum principle for pointwise solutions.\\

\noindent
{\bf Notations.} Throughout the note,  we shall use the  following notations.  We let $\omega_{N-1}= {2\pi^{\frac{N}{2}}}/{\Gamma(\frac{N}{2})} $ denote the measure of the unit sphere in $\R^N$ and,  for a set $A \subset\R^N$ and $x \in \R^N$, we define $\delta_A(x):=\dist(x,A^c)$ with $A^c=\R^N\setminus A$ and, if $A$ is measurable, then $|A|$ denotes its Lebesgue measure. Moreover, for given $r>0$, let $B_r(A):=\{x\in \R^N\;:\; \dist(x,A)<r\}$, and let $B_r(x):=B_r(\{x\})$ denote the ball of radius $r$ with $x$ as its center. If $x=0$ we also write $B_r$ instead of $B_r(0)$.
If $A$ is open, we denote by $C^k_c(A)$ the space of function $u:\R^N\to \R$ which are $k$-times continuously differentiable and with support compactly contained in $A$. If $f$ and $g$ are two functions,
then, $f\sim g$ as $x\to a$ if  $\frac{f(x)}{g(x)}$ converges to a constant as $x$ converges to $a$.
By $\cS(\R^N)$, we denote the  Schwartz space  of rapidly decaying functions in $\R^N$  with the topology   generated by the family of semi-norm
\[ 
P_N(\phi) = \sup_{x\in \R^N}(1+|x|)^N\sum_{|\alpha|\le N}|D^{\alpha}\phi(x)|, \quad N=0,1,2,\dots
\]
with $\phi\in \cS(\R^N)$. For any $\phi\in \cS(\R^N)$, we define the Fourier transform of $\phi$ by
\[
\widehat\phi(\xi)=\cF(\phi)(\xi) = \frac{1}{(2\pi)^{\frac{N}{2}}}\int_{\R^N}e^{-ix\cdot\xi}\phi(x)\ dx.
\] 
 Let  $\cS'(\R^N)$ be the set of all tempered distributions,  the topological dual of  $\cS(\R^N)$. The Fourier Transform  can be extended continuously from  $\cS(\R^N)$ to  $\cS'(\R^N)$.

% \textbf{Acknowledgements.}~ 
% This work was initiated when the author visted Rutgers University of New-Jesey, USA, through  the Abbas
% Bahri Excellence Fellowship.  He is greatful to the
% Department of Mathematics for the hospitality and wishes to thank  Pr. Yan Yan Li, Pr. Zheng-chao Han, and Pr. Denis Kriventsov for their helpful discussions. The author is  partially funded by
% a Research Strategic Initiatives Grant (RSIG22) from the University of Northern British Columbia, Canada.\\

\section{Definition and properties of  the  operator } \label{sECT1}

 Let start with the well-known  fractional Laplace operator $(-\Delta)^s$ with $s\in (0,1)$, appearing in the definition of $(I+(-\Delta)^s)^{\log}$.  Due the important links to stochastic processes and partial differential equations,   the  fractional Laplacian   has received by far  most attention see e.g. \cite{JW20, FJW22,V18,BG60,BJ07,AL22} and the references therein. It is defined  for a compactly supported functions $u : \R^N \to \R$ of class $C^2$ via the singular integral 
\begin{equation*}
\begin{split}
 (-\Delta)^s u(x)&=C_{N,s}\lim_{\epsilon\to 0^+}\int_{\R^N\setminus B_{\epsilon}(x)}\frac{u(x)-u(y)}{|x-y|^{N+2s}}\ dy 
 \end{split}
\end{equation*}
 where  the normalization constant $C_{N,s}$ is given by
\begin{equation}\label{const}
C_{N,s}:= \pi^{-\frac{N}{2}}2^{2s}s\frac{\Gamma(\frac{N}{2}+s)}{\Gamma(1-s)}, 
\end{equation}
and it is chosen such that, equivalently, the fractional Laplacian is defined by its Fourier transform
  $\cF\bigl((-\Delta)^su\bigr) = |\cdot|^{2s}\cF{u} $.
It is well known that for that for  $u\in \cC^{2}_c(\R^N)$, the following limits hold
\begin{align*}
\lim_{s\to 1^-}(-\Delta)^su(x) = -\Delta u(x)\quad\text{ and }\quad \lim_{s\to 0^+}(-\Delta)^su(x) =  u(x).
\end{align*}
 We also observe that $\lambda^{\alpha}$,  $\alpha>0$ is a  Bernstein functions with representation 
 \[
 \lambda^{\alpha}= \frac{1}{|\Gamma(-\alpha)|}\int_{0}^{\infty}(1-e^{-t\lambda}) t^{-1-\alpha}\ dt.
 \]
 Moreover,  it is not difficult to check using the Cauchy-Frullani's integral formula \cite{A90} that $\log(1+\lambda^{\alpha})$ has the representation 
 \[\log(1+\lambda^{\alpha}) = \int_{0}^{\infty}(1-e^{-t\lambda^{\alpha}})\frac{e^{-t}}{t}\ dt,
 \]
and it is in fact  a Bernstein function too.  Therefore, using Bochner's subordination  (by substituting formally $\lambda =-\Delta$),   the fractional   logarithmic  Schr\"odinger operator can also be represented in the sense of spectral theory in $L^2$ by
 \[
 (I+(-\Delta)^s)^{\log} =  \int_{0}^{\infty}(1-e^{-t(-\Delta)^{s}})\frac{e^{-t}}{t}\ dt.
 \]
 It well known from the spectral theorem  that the fractional Laplacian $(-\Delta)^s$ generates a strongly continuous semigroup of operator $P_t=e^{-t(-\Delta)^s}$, which  is the Fourier multiplier with symbol $e^{-t|\xi|^{2s}}$. Hence $P_t$  is the convolution operator with a symmetric kernel function $p_s(z,t)$ satisfying the Fourier representation $$\cF(p_s(\cdot,t))(\xi)=e^{-t|\xi|^{2s}}.$$ The function $p_s(t,x)$  is the transition density of a  $s$-stable L\'evy process.   
Moreover,  for  two values of the parameter $s$, namely $s=1$ and $s=\frac{1}{2}$, the  function $p_s(x,t)$ is explicitly known and in the first case, it writes as 
$$p_1(x,t)=(4\pi t)^{-\frac{N}{2}}e^{-\frac{|x|^2}{4t}}, \quad t>0,$$ 
known  as the Gaussian kernel for the standard heat equation (see \cite{V18},\cite[Page 180]{LL97}). For the fractional case  $s=\frac{1}{2}$, we have that
\[
p_{\frac{1}{2}}(x,t)=\gamma_N\frac{t}{(t^2+|x|^2)^{\frac{N+1}{2}}}\quad\text{ with }\quad \gamma_N=\Gamma(\frac{N+1}{2})\pi^{-\frac{N+1}{2}}.
\]
For $s\in (0,1)$, by the Fourier inversion theorem, $p_s(z,t)$ takes the form
\[
p_s(z,t)=\frac{1}{(2\pi)^{\frac{N}{2}}}\int_{\R^N}e^{-i\xi\cdot z} e^{-t|\xi|^{2s}} d\xi.
\]	
Also $\cF(p_s)(\xi)=e^{-t|\xi|^{2s}}$ is rapidly decreasing and therefore $p_s(z,t)$ is of class $\cC^{\infty}$. We also have the following scaling property  
\[
p_s(z,t) = t^{-\frac{N}{2s}}p_s(t^{-\frac{1}{2s}}z,1) 
\]
and the following limits 
\[
 \lim_{|z|\to \infty}|z|^{N+2s}p_s(z,1) = C_{N,s}\qquad\text{ and }\qquad  \lim_{t\to 0}\frac{|z|^{N+2s}p_s(z,t)}{t} = \kappa_{N,s}.
\]
Moreover,  with the constant $C_{N,s}>0$ as in  \eqref{const} with $s\in (0,1)$, it is well known that for any $t>0$   the transition density functions $p_s(y,t)$ satisfies (see \cite{BG60,BJ07,V18})
\begin{equation}\label{Transsition-inequality}
p_{s}(y,t)\sim C_{N,s}\min\Big( \frac{t}{|y|^{N+2s}}, t^{-\frac{N}{2s}}\Big),\quad y\in \R^N.
\end{equation}
Moreover, $\frac{t}{|y|^{N+2s}}\le t^{-\frac{N}{2s}}$ if and only if $t\le |y|^{2s}$.  Since we have that
$$\text{$P_tu(x)=e^{-t(-\Delta)^{s}}u(x)= [p_s(\cdot,t)*u](x)$\qquad and \qquad $p_s(-y,t)=p_s(y,t)$},$$ 
for $u\in \cS(\R^N)$, we can write 
\begin{align*}
 (I+(-\Delta)^s)^{\log}u(x) &=  \int_{0}^{\infty}\big(u(x)-P_tu(x))\frac{e^{-t}}{t} dt\\
 &=\int_{0}^{\infty}(u(x)-[p_s(\cdot,t)*u](x)\big)\frac{e^{-t}}{t} dt\\
 &=\int_{0}^{\infty}(u(x)-[p_s(-\cdot,t)*u](x)\big)\frac{e^{-t}}{t} dt,
 \end{align*}
where $P_tu$ can be interpreted as the solution of the fractional heat equation with initial datum $u$, namely, $P_tu$ is the solution of the fractional heat equation
\[
\begin{cases}
\partial_{t}P_tu=(-\Delta)^sP_tu,\quad \quad t>0, \\
 P_0u=u.
\end{cases}
\]
Finally, using the  density property of $p_s(\cdot,t)$, the fractional   logarithmic  Schr\"odinger operator can be represented as
\begin{align*}
 (I+(-\Delta)^s)^{\log}u(x) &=  \frac{1}{2}\int_{0}^{\infty}\big(2u(x)-[p_s(\cdot,t)*u](x)-[p_s(\cdot,t)*u](x)\big)\frac{e^{-t}}{t} dt\\
 &=\frac{1}{2}\int_{\R^N}\big(2u(x)-u(x+y)- u(x-y)\big)K_s(y) \ dy,
\end{align*}
where the kernel function $K_s:\R^N\to\R$ is given by
\begin{equation}\label{KerenelExp}
\begin{split}
K_s(z) &= \int_{0}^{\infty}p_s(z,t)\frac{e^{-t}}{t}\ dt\\
&=\frac{1}{(2\pi)^{\frac{N}{2}}}\int_{0}^{\infty}\int_{\R^N}e^{-i\xi\cdot z} e^{-t|\xi|^{2s}} \frac{ e^{-t}}{t}\ d\xi dt.
\end{split}
\end{equation}
The explicit expression of the kernel $K_s$ is unknown due to the difficulty on computing the integral appearing in \eqref{KerenelExp}. Moreover, from the explicit expression of the density function  $p_{1}(\cdot,t)$, the explicit expression for $K_1$ is known (see \cite{F22}) and it is given in term of the modified Bessel function of second kind. The corresponding operator, denoted by  $(I-\Delta)^{\log}$ is the  logarithmic  Schr\"odinger operator.  It has been proved in  \cite{F22} that for  compactly supported Dini continuous functions $u : \R^N \to \R$, the operator $(I-\Delta)^{\log}$ has  a singular integral representation  given  by
\begin{equation}\label{log-s-1}
(I-\Delta)^{\log} u(x)= \lim_{\epsilon\to 0}\int_{\R^N\setminus B_{\epsilon(x)}}(u(x)-u(y))K_1(x-y)\ dy,
\end{equation}
where   $K_1$ is given by
\[
 K_1(x-y)= 2^{2^{1-\frac{N}{2}}}\pi^{-\frac{N}{2}}|x-y|^{-\frac{N}{2}}\kappa_{\frac{N}{2}}(x-y)
\]
and  $\kappa_{\nu}$ is the modified Bessel function of second kind with index $\nu>0$, given by
\[
\kappa_{\nu}(r)=\frac{(\pi/2)^{\frac{1}{2}}r^{\nu}e^{-r}}{\Gamma(\frac{2\nu+1}{2})}\int_{0}^{\infty}e^{-rt}t^{\nu-\frac{1}{2}}(1+{t}/{2})^{\nu-\frac{1}{2}}dt.
\] 
From the radial property of the  density function  $p_{s}(\cdot,t)$, $s\in(0,1]$, it easy to see that the kernel is symmetric, i.e. it satisfies $K_s(-z)=K_s(z)$ for all $z\in \R^N$ and 
\begin{equation}\label{non-integrability}
\int_{\R^N}K_s(z)\ dz=\infty.\quad
\end{equation}
The above condition in \eqref{non-integrability} is visible from the asymptotics of the kernel $K_s$ in Theorem \ref{properties1}. This shows that for any $s>0$ the  the kernel $K_s$ is singular at the origin. Therefore, the fractional   logarithmic  Schr\"odinger operator is a singular integral operator and  should be understood in the principal value sense,
\begin{equation}\label{Intro-int-1}
\begin{split}
(I+(-\Delta)^s)^{\log} u(x)&=P.V\int_{\R^N}\big(u(x)-u(y)\big)K_s(x-y) \ dy\\
&=\lim_{\epsilon\to 0}\int_{\R^N\setminus{B_{\epsilon}(x)}}\big(u(x)-u(y)\big)K_s(x-y) \ dy.
\end{split}
\end{equation} 
Moreover, for $\Omega$ open set of $\R^N$, one may also define the \textit{regional}  fractional   logarithmic  Schr\"odinger operator by
\[
(I+(-\Delta)^s)_{\Omega}^{\log} u(x)=\lim_{\epsilon\to 0}\int_{\Omega\setminus{B_{\epsilon}(x)}}\big(u(x)-u(y)\big)K_s(x-y) \ dy,
\]
which may might  be interpreted as the  regional Fractional Laplacian $(-\Delta)^s_{\Omega}$, (see \cite{C18}). In particular,  for $s=1$ and up to a multiplicative constant, from the asymptotics of the kernel $K_s$ (see Theorem \ref{properties1}), an example of  regional     logarithmic  Schr\"odinger operator can be identified with the operator $H$ in \cite{FKT20} given by
\[
(I-\Delta)_{\mathbb{S}^{N}}^{\log} u(x):=Hu(x)=P. V. \int_{\mathbb{S}^N}\frac{u(x)-u(y)}{|x-y|^N}\ dy,
\]
where $\mathbb{S}^N:\{x\in\R^{N+1}: \|x\|=1\}$ is the unit sphere in $\R^{N+1}$. The authors in \cite{FKT20} classified all nonnegative solutions $u$ of the equation
\[
(I-\Delta)_{\mathbb{S}^{N}}^{\log} u= u\log u\quad \text{ in }\quad\mathbb{S}^N.
\]
  If $\Omega$ has small volume, the regional  fractional   logarithmic  Schr\"odinger operator up to multiplicative constant also coincides with the regional logarithmic Laplacian \cite{TW22},
\[
(I-\Delta)_{\Omega}^{\log} u(x):=L_{\Delta}^{\Omega}u(x)=P. V. \int_{\Omega}\frac{u(x)-u(y)}{|x-y|^N}\ dy.
\]
We note that in the particular case $N = 1$ and $s=1$, it follows from the definition in \eqref{log-s-1} that (see \cite{GR12,B14,KS13,LP05})
\begin{equation}\label{log-1}
(I-\Delta)^{\log} u(x)= P.V. \int_{\R^N}\frac{u(x)-u(y)}{|x-y|}e^{-|x-y|}\ dy.
\end{equation}
The main   result of this section  provides  an asymptotics property  of the kernel $K_s$ at the singularity and at infinity.  Since the explicit expression of $K_s$ is not known for $s\in (0,1)$, the following theorem is capital and will be used almost everywhere in this note. It writes as follows.

\begin{thm}
\label{properties1}
\begin{enumerate}[(i)]
    \item 
 Let  $C_{N,s}>0 $ be as  given  in \eqref{const} with the parameter $s\in (0,1)$. Then the  kernel function $K_s$ for $(I+(-\Delta)^s)^{\log} $ satisfies the asymptotics
\begin{equation}\label{Asymptoti-o}
K_s(z)\sim \begin{cases} \kappa_{N,s}|z|^{-N} \qquad\quad \text{ as} \quad |z|\to 0\\
\\
C_{N,s}|z|^{-N-2s} \qquad \text{ as} \quad |z|\to \infty.
\end{cases}
\end{equation}
\item
For $s\in (0,1]$ and $u\in \cC^{1}_c(\R^N)$, the operator $(I+(-\Delta)^s)^{\log} $ can represented via its Fourier transform as
\begin{equation}
\cF((I+(-\Delta)^s)^{\log} u)(\xi) = \log(1+|\xi|^{2s})\cF(u)(\xi), \qquad \xi\in \R^N.
\end{equation}
\item 
Let $u \in \cC^{1}_c(\R^N)$. Then, the function $s\mapsto (I+(-\Delta)^s)^{\log}u$ is uniformly continuous. In particular,  the following limits hold
\begin{align*}
&\lim_{s\to 1}(I+(-\Delta)^s)^{\log}u = (I-\Delta)^{\log}u\qquad\text{and}\qquad \lim_{s\to 0^+}(I+(-\Delta)^s)^{\log}u = 0.
\end{align*}
where the operator $\logrel u$ is defined in \eqref{log-s-1}.
\item The operator $\logrels$ is invariant under $(\lambda+(-\Delta)^s)^{-1}$, $\lambda>0$. That is 
\begin{equation}\label{invariant}
    \logrels\big[(\lambda+(-\Delta)^s)^{-1}u\big]= (\lambda+(-\Delta)^s)^{-1}\logrels u
\end{equation}
for all   $u\in \cC^1_c(\R^N).$
\end{enumerate}
\end{thm}

\begin{proof}
$(i).$ This can be directly checked in particular for  $s=\frac{1}{2}$, using the explicit expression of $p_{\frac{1}{2}}(x-y,t)$. Indeed, by a change of variable, we have that
\begin{equation}\label{chagevar}
\begin{split}
K_{\frac{1}{2}}(x-y)&=\frac{\gamma_N}{|x-y|^{N}}\int_0^{\infty}\frac{e^{-t|x-y|}}{(t^2+1)^{\frac{N+1}{2}}}\ dt\\
&=\frac{\gamma_N}{|x-y|^{N+1}}\int_0^{\infty}\frac{e^{-t}}{(\frac{t^2}{|x-y|^2}+1)^{\frac{N+1}{2}}}\ dt.
\end{split}
\end{equation}
Using the first equality of $K_{\frac{1}{2}}$ in \eqref{chagevar}, we have
\[
\lim_{|x-y|\to 0}|x-y|^N K_{\frac{1}{2}}(x-y)=\frac{\Gamma(\frac{N}{2})}{2\pi^{N/2}}.
\]
Using the second  equality of $K_{\frac{1}{2}}$ in \eqref{chagevar}, we get
\begin{align*}
\lim_{|x-y|\to \infty}|x-y|^{N+1} K_{\frac{1}{2}}(x-y)=\gamma_N.
\end{align*}
For  general $s\in(0,1)$,  we shall use the following representation of the incomplete gamma function  (see \cite[Page 177]{OLBC10}) given by
\begin{equation}\label{imcompleteGamma}
\Gamma(a,z)=z^ae^{-z}\int_0^{\infty}\frac{e^{-tz}}{(1+t)^{1-a}}\ dt, \quad \text{ for } ~z>0.
\end{equation}
Moreover,  with the constant $C_{N,s}>0$ as in  \eqref{const} with $s\in (0,1)$, it is well known that for any $t>0$   the transition density functions $p_s(x,y,t)$ satisfies (see \cite{BG60,BJ07,V18})
\begin{equation}\label{Transsition-inequality2}
p_{s}(x,y,t)\sim C_{N,s}\min\left(t^{-\frac{N}{2s}}, \frac{t}{|x-y|^{N+2s}}\right),\quad x,y\in \R^N.
\end{equation}
However, $\frac{t}{|x-y|^{N+2s}}\le t^{-\frac{N}{2s}}$ if and only if $t\le |x-y|^{2s}$. 
Thus, using \eqref{Transsition-inequality2}, we have
\begin{align*}
K_s(x-y)&=\int_{0}^{\infty}p_s(x-y,t)e^{-t}\frac{dt}{t}\\
&\sim C_{N,s}\int_{0}^{\infty}\min\left(t^{-\frac{N}{2s}}, \frac{t}{|x-y|^{N+2s}}\right)e^{-t}\frac{dt}{t}\\
&= C_{N,s}\int_{0}^{|x-y|^{2s}}\frac{e^{-t}}{|x-y|^{N+2s}}\ dt+C_{N,s}\int_{|x-y|^{2s}}^{\infty}t^{-\frac{N+2s}{2s}} e^{-t}\ dt\\
&= \frac{C_{N,s}}{|x-y|^{N+2s}}\Big(1-e^{-|x-y|^{2s}}\Big)+C_{N,s}\Gamma(-\frac{N}{2s},|x-y|^{2s})\\
&=\frac{C_{N,s}}{|x-y|^{N}}\Big(\frac{1-e^{-|x-y|^{2s}}}{|x-y|^{2s}}\Big)+\frac{C_{N,s}e^{-|x-y|^{2s}}}{|x-y|^{N}} \int_0^{\infty}\frac{e^{-t|x-y|^{2s}}}{(1+t)^{1+\frac{N}{2s}}}\ dt.
\end{align*}
It follows by the dominated convergent theorem that
\begin{align*}
\lim_{|x-y|\to 0}\frac{K_s(x-y)}{|x-y|^{N}}&= C_{N,s}+C_{N,s}\int_0^{\infty}\frac{1}{(1+t)^{1+\frac{N}{2s}}}\ dt\\
&=C_{N,s}\left(1+\frac{\sqrt{\pi}\Gamma(\frac{N+s}{2s})}{2\Gamma(\frac{N+2s}{2s})}\right)=:\kappa_{N,s}
\end{align*}
and 
\begin{align*}
\lim_{|x-y|\to +\infty}\frac{K_s(x-y)}{|x-y|^{N+2s}}&=C_{N,s}\lim_{|z|\to +\infty}\Big(\big(1-e^{-|z|^{2s}}\big)+\frac{|z|^{2s}}{e^{|z|^{2s}}} \int_0^{\infty}\frac{e^{-t|z|^{2s}}}{(1+t)^{1+\frac{N}{2s}}} dt\Big)\\
&=C_{N,s}.
\end{align*}
This completes the proof of $(i)$.

$(ii).$ Let  $s\in (0,1]$. Since $u\in \cC^1_c(\R^N)$, the integrand in \eqref{Intro-int-1} is a $L^1$ function on $\R^N\times\R^N$. Thus,   one applies Fourier transform on both side of \eqref{Intro-int-1} combined with the Fubini's theorem and the symmetry property of the kernel $K_s$ to see that
\begin{align*}
&\cF((I+(-\Delta)^s)^{\log} u)(\xi)=\cF(u)(\xi)\int_{\R^N}(1-e^{-iy.\xi})K_s(y)\ dy\\
&\qquad\qquad\qquad=\cF(u)(\xi)\Big(\int_0^{\infty}\int_{\R^N}\Big(p_s(y,t) -e^{iy\cdot(-\xi)}p_s(y,t)\Big)\ dy\  e^{-t} \frac{dt}{t}\Big)\\
&\qquad\qquad\qquad=\cF(u)(\xi)\int_0^{\infty}\Big(1-e^{-t|\xi|^{2s}}\Big)\  e^{-t} \frac{dt}{t}\\
&\qquad\qquad\qquad=\log(1+|\xi|^{2s})\cF(u)(\xi).
\end{align*}
This completes the proof of item $(ii)$ 

$(iii).$ By the Fourier inversion theorem, we use the above representation obtained   (ii) given by
$$(I+(-\Delta)^s)^{\log}u(x)= \cF^{-1}\big( \log(1+|\xi|^{2s})\cF(u)(\xi)\big)(x).$$
Let $s_0>0$. We bound the difference using  the fundamental theorem of calculus, 
\begin{align*}
&\Big|(I+(-\Delta)^s)^{\log}u(x)-(I+(-\Delta)^{s_0})^{\log}u(x)\Big|\\
&\hspace{2cm}= \Big|\int_{\R^N}e^{-ix\cdot \xi}\Big(\log(1+|\xi|^{2s})-\log(1+|\xi|^{2s_0})\Big)\cF(u)(\xi)\ d\xi\Big|\\
&\hspace{2cm}\le  \int_{\R^N}\Big|\log(1+|\xi|^{2s})-\log(1+|\xi|^{2s_0})\Big||\cF(u)(\xi)|\ d\xi\\
&\hspace{2cm}\le2|s-s_0|\int_{\R^N}\int_0^1\frac{|\xi|^{2\tau(s-s_0)+2s_0}|\log|\xi||}{1+|\xi|^{2\tau(s-s_0)+2s_0}}|\cF(u)(\xi)|d\tau\ d\xi\\
&\hspace{2cm}\le 2|s-s_0|\int_{\R^N}\int_0^1|\log|\xi|||\cF(u)(\xi)|d\tau\ d\xi\\
&\hspace{2cm}\le \frac{2|s-s_0|}{\epsilon}\int_{\R^N}\big((|\xi|\chi_{B_1(0)})^{-\epsilon}+(|\xi|\chi_{\R^N\setminus B_1(0)})^{\epsilon}\big)|\cF(u)(\xi)|\ d\xi,
\end{align*}
where $\epsilon>0$ is a small parameter coming from  the inequality (see \cite{JSW20})
\begin{equation}\label{log-ineq}
\log \rho\le \frac{1}{\epsilon}\big((\rho\chi_{\{\rho<1\}})^{-\epsilon}+(\rho\chi_{\{\rho\ge 1\}})^{\epsilon}\big),\qquad \rho>0.
\end{equation}
It  follows that
\[
\begin{split}
 &\Big|(I+(-\Delta)^s)^{\log}u(x)-(I+(-\Delta)^{s_0})^{\log}u(x)\Big|\\
 &\qquad\qquad\qquad\le \frac{2|s-s_0|}{\epsilon} |(-\Delta)^{-\frac{\epsilon}{2}}(u\chi_{B_1})|  +|(-\Delta)^{\frac{\epsilon}{2}}(u\chi_{\R^N\setminus B_1})|  \\
 &\qquad\qquad\qquad\le \frac{2|s-s_0|}{\epsilon}\|u\|_{\cC^1_c(\R^N)}.
\end{split}
\]
This shows that  the function $s\mapsto (I+(-\Delta)^s)^{\log}u$ is Lipschitz and therefore uniformly continuous and thus the proof of $(iii)$

$(iv).$ Set $w:= (\lambda+(-\Delta)^s)^{-1}u$, then by definition
\[
\begin{split}
(I+(-\Delta)^s)^{\log}w(x)&=\int_{\R^N}(w(x)- w(y))k(x-y)\ dy\\
&=(\lambda+(-\Delta)^s)^{-1}\int_{\R^N}(u(x)- u(y))k(x-y)\ dy\\
&=(\lambda+(-\Delta)^s)^{-1}\logrels u(x).
\end{split}
\]
This shows $(iv)$ and hence the proof of Theorem \ref{properties1}.
\end{proof}

\begin{remark}
\begin{enumerate}[(a)]
    \item 
 Note that the asymptotics of $K_s$ in \eqref{Asymptoti-o}  as $|x|\to \infty$ holds only for $0<s<1$.  In the particular case $s=1$, the behavior of $K_1$ at infinity is different and decays exponentially. Namely
\[
K_1(z)\sim  \kappa_{N}|z|^{-N}  \quad \text{ as} \quad |z|\to 0  
\]   
and 
\[
K_1(z)\sim C_N |z|^{-\frac{N+1}{2}}e^{-|z|} \quad \text{ as} \quad |z|\to \infty.
\]   
We refer to  \cite{F22} where the proof for $s=1$ was done using the properties of the modified Bessel function of second kind $\kappa_{\nu}$ defined above.
\item  The asymptotics of the kernel $K_s$ in Theorem \ref{properties1} shows  that the operator $\logrels$ has the same singular local behavior as that of the logarithmic Laplacian $\loglap$ (see \cite{CW19,FJW22,HS22,CV23}) and it  is  comparable   to  the fractional Laplacian at infinity.  We recall that the  logarithmic Laplacian $\loglap$ is  the pseudodifferential operator with Fourier symbol $2\log|\cdot|$ and  for  compactly supported Dini continuous functions $\phi : \R^N \to \R$,  it is  pointwisely defined by 
\begin{equation}\label{logarithmic}
\loglap\phi(x) = c_N\lim_{\epsilon\to 0}\int_{\R^N\setminus B_{\epsilon}(x)}\frac{\phi(x)1_{B_1(x)}(y)-\phi(y)}{|x-y|^N}\ dy +\rho_N\phi(x),
\end{equation}
where the constants  $c_{N}:=\frac{\Gamma(N/2)}{\pi^{N/2}}$ and $\rho_N:=2\ln 2+\psi(\frac{N}{2})-\gamma$, see \cite{CW19} for more details. 
\end{enumerate}
\end{remark}\noindent
It follows from  \eqref{Asymptoti-o} that   the kernel $K_s$ satisfies the L\'evy integrability property
\begin{equation}\label{integrability} 
\int_{\R^N}\min\{1,|z|^{\epsilon}\}K_s(z)\ dz<\infty, \quad (z\in\R^N)
\end{equation}
for any small enough $\epsilon>0$. This clearly show that the operator $\logrels$ belongs to the class of nonlocal operators with small order (see \cite{FJ22, EA18,HS22, CS22}). We  recall below the definition of  the \textit{order of an operator}, see  also \cite[Definition 2.1.2]{GM02} and  \cite{PF22}. 
\begin{defi}
    Let $L_K$ be an integro-differential operator with kernel $K$. We  define the order of $L_K$ as the infimum of the value $\sigma>0$ for which the following holds
    \[
\int_{\R^N}\min\{1,|x-y|^{\sigma}\}K(x-y)\ dy<\infty,\quad (x\in\R^N).
    \]
\end{defi}\noindent
We next give the definition  of Dini continuous functions. This provides the weak regularity that one can put on a function $u$ for the quantity $(I+(-\Delta)^s)^{\log}u(x)$ to be well-defined for  $x\in \R^N$. Let  $u:\Omega\to \R$ be a measurable function. We  introduce the modulus of continuity   $\omega_{u,x,\Omega}: (0, +\infty)\to [0,+\infty)$ of $u$  at a point $x\in \Omega$, defined by 
\[
  \omega_{u,x,\Omega}(r)=\sup_{\substack{y\in \Omega,\ |x-y|\le r}}|u(x)-u(y)|.
\]
Then, a  function $u$ is called Dini continuous at $x$ if 
$$\int_{0}^1\frac{\omega_{u,x,\Omega}(r)}{r}\ dr<\infty.$$ 
We call $u$ uniformly Dini continuous in $\Omega$ for the uniform modulus of continuity 
\[
\omega_{u,\Omega}(r):= \sup_{x\in \Omega}\omega_{u,x,U}(r)\quad\text{ if }\qquad \int_{0}^1\frac{\omega_{u,\Omega}(r)}{r}\ dr<\infty.
\]
Therefore, for  compactly supported Dini continuous functions $u : \R^N \to \R$, the fractional logarithmic Schr\"odinger operator $\logrels u(x)$ is  pointwisely well-defined.\\
We introduce also  the weighted  space $\cL^s(\R^N)$, $s\in (0,1)$  defined by
\[
 \cL^s(\R^N):=\{u\in L^1(\R^N): \|u\|_{\cL^s}<\infty\},
\]
where
\[
 \|u\|_{\cL^s(\R^N)}:=\int_{\R^N}\frac{|u(y)|}{(1+|y|)^{N+2s}}\ dy.
\]
In the following proposition, we list some properties of the operator $(I+(-\Delta)^s)^{\log}$. Some items of the  proposition can be proved using the same computational arguments as in \cite[Proposition 2.1]{F22}. We shall provide  only the proof of item $(2)$ and $(3)$ for completeness.

\begin{prop}
\label{well-defined}
\begin{enumerate}[(1)]
\item Let $u\in \cL^s(\R^N)\cap L^{\infty}(\R^N)$. If $u$ is locally  Dini continuous at some point $x\in \R^N$, then the  operator $(I+(-\Delta)^s)^{\log}u$~ is pointwisely defined by  
\[
(I+(-\Delta)^s)^{\log}u(x)=\int_{\R^N}(u(x)- u(y))k(x-y)\ dy.
\]
\item  Let $\phi\in C_c^{\alpha}(\R^N)$ for some $\alpha\in (0,1)$,  there is $C=C(N,s,\alpha,\phi)$ such that 
\[
\Big|(I+(-\Delta)^s)^{\log}\phi(x)\Big|\le C\frac{\|\phi\|_{C^{\alpha}(\R^N)}}{(1+|x|)^{{N+2s}}}.
\] 
In particular, for $u\in \cL^s(\R^N)$,~ $(I+(-\Delta)^s) u$~ defines a distribution via the map 
\[
\phi\mapsto\langle (I+(-\Delta)^s)^{\log} u,\phi\rangle=\int_{\R^N} u(I+(-\Delta)^s)^{\log}\phi\ dx.
\]
\item If $u\in C^{\beta}(\R^N)$ for some $\beta>0$, then $(I+(-\Delta)^s)^{\log}u\in  C^{\beta-\epsilon}(\R^N)$ for every $\epsilon$ such that $0<\epsilon<\beta$  and there exists a constant $C:=C(N,s,\beta,\epsilon)>0$ such that
\[
\|(I+(-\Delta)^s)^{\log}u\|_{\cC^{\beta-\epsilon}(\R^N)}\le C\|u\|_{C^{\beta}(\R^N)}.
\]
\item Let $\phi,\psi\in\cC_c^{\infty}(\Omega)$. Then we have the product rule  
\begin{align*}
(I+(-\Delta)^s)^{\log}(\phi\psi)(x)&=\phi(x)(I+(-\Delta)^s)^{\log}\psi(x) + \psi(x)(I+(-\Delta)^s)^{\log}\phi(x)\\
&\qquad-\int_{\R^N}(\phi(x)-\phi(y))(\psi(x)-\psi(y))k(x-y)\ dy.
\end{align*}
% with 
% $$
% \Lambda(\phi,\psi):=\int_{\R^N}(\phi(x)-\phi(y))(\psi(x)-\psi(y))k(x-y)\ dy.
% $$
If $(\rho_{\epsilon})_{\epsilon>0}$ is a  family of  mollified, then
\[
[(I+(-\Delta)^s)^{\log}(\rho_{\epsilon}\ast \phi)](x) = \rho_{\epsilon}\ast[ (I+(-\Delta)^s)^{\log}\phi](x).
\] 
\end{enumerate}
\end{prop}

\begin{proof}[Proof of Proposition \ref{well-defined}]
We provide only the proofs for $(2)$ and $(3)$, see \cite{F22} for the rests.

Proof of $(2)$. Let $\phi\in C_c^{\alpha}(\R^N)$. We use the   representation 
\begin{align*}
(I+(-\Delta)^s)^{\log}\phi(x)=\frac{ 1}{2}\int_{\R^N}{\big(2\phi(x)-\phi(x+y)- \phi(x-y)\big)}K_s(x-y)\ dy.
\end{align*}
Put $A:= \|\phi\|_{C^{\alpha}(\R^N)}$. Noticing that $\phi\in C_c^{\alpha}(\R^N)$, we have 
$$
|2\phi(x)-\phi(x+y)-\phi(x-y)|\le A\min\{1,|y|^{\alpha}\}.
$$
Therefore, for any $x\in \R^N$, we have from \eqref{Asymptoti-o} with $0<r<1$ that  
\begin{align*}
&|(I+(-\Delta)^s)^{\log}\phi(x)|\le\frac{ 1}{2}\int_{\R^N}{|2\phi(x)-\phi(x+y)- \phi(x-y)|}K_s(y)\ dy\\
&\le  C_NA\Big(\int_{B_r}|y|^{\alpha-N}\ dy+\int_{B_1\setminus B_r}\frac{1}{|y|^N}\ dy+\int_{\R^N\setminus B_1}\frac{1}{|x-y|^{N+2s}}\ dy \Big)\\
&\le C(N,s,r,\alpha).
\end{align*} 
Next,  Let $R\ge 1$ be such that   $\supp\ \phi\subset B_{R}(0)$.  Let $x\in \R^N$ satisfying  $\frac{|x|}{4}>R$ such that $\phi(x)=0$. Then,  for $y\in B_{R}(0)$,
$$\text{  $|x-y|\ge \frac{|x|}{2}+\frac{|x|}{2}-|y|\ge \frac{|x|}{2}+1\ge \frac{1}{2}(|x|+1)$.}$$
Moreover, since $\phi(x)\equiv 0$ for $x\in \R^N\setminus B_R(0)$, it follows  that
\begin{align*}
|(I+(-\Delta)^s)^{\log}\phi(x)|&\le  C_{N,s}A\int_{\supp \ \phi}\frac{1}{|x-y|^{N+2s}}\ dy \le \frac{C_{N,s}|\supp\ \phi|A}{(1+|x|)^{{N+2s}}}.
\end{align*}
Therefore, combining the above computations, we find with $C:=C(N,s,\alpha,\phi)$ that 
\[
|(I+(-\Delta)^s)^{\log}\phi(x)|\le \frac{C}{(1+|x|)^{N+2s}}\qquad \text{ for all }   x\in \R^N.
\]
From the above computations, we have with  $C:=C(N,s,\alpha,\phi)$ that 
$$|\langle \logrels u,\phi\rangle|\le C\|\phi\|_{C^{\alpha}(\R^N)}\|u\|_{\cL^s(\R^N)}.$$
Moreover,  if  a sequence $\{u_n\}_n$ converges to $u$ in $\cL^s(\R^N)$ as $n\to \infty$ then, 
$$
|\langle \logrels u_n-\logrels u,\phi\rangle|\le C_{N,\phi}A\|u_n-u\|_{\cL^s(\R^N)}\to 0 \quad\text{ as } n\to \infty.
$$
This completes the proof of $(2)$. 

Proof of $(3)$. Let $0<r<1$.  We have the following  estimate  of the difference,
\begin{align*}
|\logrels u(x_1)-\logrels u(x_2)|\le I_1+I_2
\end{align*}
where $I_1$ and $I_2$ are given by
\begin{align*}
I_1& := \int_{B_r} \big(|u(x_1)-u(x_1+y)|+|u(x_2)-u(x_2+y)|\big)K_s(y)\ dy\\\
I_2& := \int_{\R^N\setminus B_r} \big(|u(x_1)-u(x_2)|+|u(x_1+y)-u(x_2+y)|\big)K_s(y)\ dy.
\end{align*}
For $I_1$, we use the inequality $|u(x_1)-u(x_1+y)|\le\|u\|_{C^{\beta}(\R^N)} |y|^{\beta}$ combined with \eqref{Asymptoti-o} to get
\[
I_1\le2\|u\|_{C^{\beta}(\R^N)}\int_{B_r}|y|^{\beta-N}\ dy\le  \frac{2\omega_{N-1}C\kappa_{N,s}}{\beta}\|u\|_{C^{k}(\R^N)}r^{\beta}.
\]
For $I_2$, we use $|u(x_1)-u(x_2)|+|u(x_1+y)-u(x_2+y)|\le 2\|u\|_{C^{\beta}(\R^N)}|x_1-x_2|^{\beta}$, 
\begin{align*} 
I_2&\le 2|x_1-x_2|^{\beta}\|u\|_{C^{\beta}(\R^N)}\Big(\int_{B_1\setminus B_r}\frac{\kappa_{N,s}}{|y|^{N}}\ dy + \int_{\R^N\setminus B_1}\frac{C_{N,s}}{|y|^{N+2s}}\ dy\Big)\\
&\le 2|x_1-x_2|^{\beta}C_{N,s}\|u\|_{C^{\beta}(\R^N)}C_{N,s}\big(\log \frac{1}{r} + 1\big)\\
&\le 2|x_1-x_2|^{\beta}C_{N,s}\|u\|_{C^{\beta}(\R^N)}\big(\frac{r^{-\epsilon}}{\epsilon} + 1\big)\\
&\le \frac{C'_{N,s}\|u\|_{C^{\beta}(\R^N)}}{\epsilon}|x_1-x_2|^{\beta} r^{-\epsilon},
\end{align*}
where we have used  \eqref{Asymptoti-o} and  the inequality $\log(\rho)\le \frac{\rho^{\epsilon}}{\epsilon}$ for $\epsilon>0$ and $\rho\ge 1$ as in \eqref{log-ineq}. Therefore, taking in particular $r=|x_1-x_2|$, we end with
 \[
 \frac{|\logrels u(x_1)-\logrels u(x_2)|}{|x_1-x_2|^{\beta-\epsilon}}\le C(N,s,\beta,\epsilon)\|u\|_{C^{\beta}(\R^N)}.
 \]
This gives the proof of $(3)$. The rest of the proof  follows from  \cite[Proposition 2.1]{F22}.
\end{proof}

\section{Fundamental solutions and  Green function}\label{Sect2}

This section introduce the fundamental solution,  the Green function for $\logrels$ and some of their properties. For $s\in(0,1]$ and $t<\frac{N}{2}$, we set
$$Q_t=(I+(-\Delta)^{s})^{-t}.$$
Then one can shows that for any $t<\frac{N}{2}$, $Q_t$ generates a strongly continuous semigroup of operator  which is the Fourier multiplier with symbol $(1+|\xi|^{2s})^{-t}$. Hence  $Q_t$ is also  the convolution operator with a symmetric kernel function $q_s(t,z)$ satisfying the Fourier representation $\cF(q_s(t,\cdot))(\xi) = (1+|\xi|^{2s})^{-t}$. The function $q_s(t,\cdot)$ is a transition density  function of a l\'evy process with the representation
\[
q_{s}(t,x,y)=\frac{1}{\Gamma(t)}\int_{0}^{\infty}p_s(x-y,\tau)\tau^{t-1}e^{-\tau}\ d\tau,\qquad (x\in R^N)
\]
One  check that $\int_{\R^N}q_s(t,x)\ dx=1$. The action of the  kernel $Q_t$ on functions  $f\in L^2(\R^N)$, is   by definition
\[
Q_tf(x)=\int_{\R^N}q_s(t,x-y)f(y)\ dy 
\]
with the Fourier transform given by
\[
\cF(Q_tf)(\xi)=(1+|\xi|^{2s})^{-t}\cF(f)(\xi).
\]
Therefore, if $u\in \cS(\R^N)$, then the function $$U(x,t):=Q_tu(x)=(I+(-\Delta)^{s})^{-t}u(x)$$ is the solution of the following ``heat equation"  for the operator $(I+(-\Delta)^{s})^{\log}$ (Cauchy problem)
\[
\begin{cases}
\partial_t U=(I+(-\Delta)^{s})^{\log}U, \quad 0<t<\frac{N}{2};\\
 U(x,t=0)=u.
\end{cases}
\]
Moreover, the function $q_s(t,x)$ is the fundamental solution to the following   Cauchy problem 
\[
\begin{cases}
\partial_{t}q_s(t,x)=\logrels q_s(t,x)\quad \text{ in }\quad\R^N\times\R_+, \\
 q_s(0,x)=\delta_0.
\end{cases}
\]
Next, consider then the Poisson problem $\logrels u=f $ in $\R^N$, then, the solution can be written in term of Green function associated to $\logrels$ by the   relation
\[
u(x)= \int_{\R^N}G(x-y)f(y)\ dy, \qquad (x\in \R^N),\qquad
\]
where the Green function $G$ is  defined by
\[
G(x)=\int_0^{\infty}q_t(x)\ dt=\int_0^{\infty} \frac{1}{\Gamma(t)}\int_{0}^{\infty}p_s(x,\tau)\tau^{t-1}e^{-\tau}\ d\tau\ dt.
\]
The main theorem of this section writes as follows.
\begin{thm}\label{Green-1}
 The Green function $G$ defined above satisfies the the asymptotics
    \begin{align*}
G(x)\sim \frac{C_{N,s}}{4s^2}\frac{1}{|x|^{N}(\log\frac{1}{|x|})^2}\qquad\text{as}\quad |x|\to 0
\end{align*}
and
\[
G(x)\sim  \frac{C_{N,s}}{\big(\log\frac{1}{|x|}\big)^2|x|^{N+2s}} \qquad\text{ as }\quad |x|\to \infty.
\]
\end{thm}

\begin{proof} The proof of the first assertion of  Lemma \ref{Green-1} can be found in \cite[ Theorem 3.2]{SSV06}. We provide a proof using different method which we also hope it is easy to follow.    Let start with the particular case $s=\frac{1}{2}$. By doing a change of variable $\tau=|x|\tau$, we get
\begin{align*}
G(x)&=\int_0^{\infty} \frac{\gamma_N}{\Gamma(t)}\int_{0}^{\infty}\frac{\tau^{t}e^{-\tau}}{(\tau^2+|x|^2)^{\frac{N+1}{2}}}\ d\tau\ dt\\
&=\frac{\gamma_n}{|x|^N}\int_0^{\infty} \frac{|x|^{t}}{\Gamma(t)}\int_{0}^{\infty}\frac{\tau^te^{-\tau|x|}}{(\tau^2+1)^{\frac{N+1}{2}}}\ d\tau\ dt.
\end{align*}
Next, set $t=\frac{t}{\log\frac{1}{|x|}}$. Using  the property  $x\Gamma(x)=\Gamma(x+1)$  for the Gamma function, we get
\begin{align*}
G(x)=\frac{\gamma_N}{|x|^N(\log\frac{1}{|x|})^2}\int_0^{\infty} \frac{te^{-t}}{\Gamma(1-\frac{t}{\log|x|})}\int_{0}^{\infty}\frac{\tau^{-\frac{t}{\log|x|}}e^{-\tau|x|}}{(\tau^2+1)^{\frac{N+1}{2}}}\ d\tau\ dt
\end{align*}
By the dominated convergent theorem, we get that
\begin{align*}
\lim_{|x|\to 0}|x|^{N}(\log\frac{1}{|x|})^2G(x)&=\int_0^{\infty}te^{-t}\ dt\int_0^{\infty} \frac{1}{(1+\tau^2)^{\frac{N+1}{2}}}\ d\tau\\
&= \gamma_N\frac{\sqrt{\pi}\Gamma(\frac{N}{2})}{2\Gamma(\frac{N+1}{2})}.
\end{align*}
 For general $s\in (0,1)$, we use the asymptotics  properties for the transition density function in \eqref{Transsition-inequality} and the representation for the  incomplete Gamma function in \eqref{imcompleteGamma}. We have
\begin{align*}
G(x)&\sim C_{N,s}\int_0^{\infty} \frac{1}{\Gamma(t)}\int_{0}^{\infty}\min\Big\{\tau^{-\frac{N}{2s}},\frac{\tau}{|x|^{N+2s}}\Big\}\tau^{t-1}e^{-\tau}\ d\tau\ dt\\
&=C_{N,s}\int_0^{\infty} \frac{1}{\Gamma(t)}\Big(\int_0^{|x|^{2s}} \frac{\tau^{t}e^{-\tau}}{|x|^{N+2s}}\ d\tau +\int_{|x|^{2s}}^{+\infty} \tau^{t-\frac{N}{2s}-1}e^{-\tau}\ d\tau\Big)\ dt\\
&= I_1+I_2.
\end{align*}
By a double changes of variables, $\tau = |x|^{2s}\tau$ and then $t= \frac{t}{2s\log\frac{1}{|x|}}$, we have
\begin{align*}
I_1&=\frac{C_{N,s}}{|x|^{N+2s}}\int_0^{\infty} \frac{1}{\Gamma(t)}\int_0^{|x|^{2s}} \tau^{t}e^{-\tau} d\tau dt\\
&=\frac{C_{N,s}}{|x|^{N}}\int_0^{\infty} \frac{|x|^{2st}}{\Gamma(t)}\int_0^{1} \tau^{t}e^{-\tau|x|^{2s}} d\tau dt\\
&=\frac{C_{N,s}}{4s^2|x|^{N}(\log\frac{1}{|x|})^2}\int_0^{1}\int_0^{\infty} \frac{te^{-t}}{\Gamma(1-\frac{t}{2s\log|x|})} \tau^{-\frac{t}{2s\log|x|}}e^{-\tau|x|^{2s}}  dt d\tau.
\end{align*}
By the dominated convergent theorem, we get
\begin{align*}
\lim_{|x|\to 0}|x|^{N}(\log\frac{1}{|x|})^2I_1=\frac{C_{N,s}}{4s^2}\int_0^{1}\int_0^{\infty} te^{-t}  dt d\tau=\frac{C_{N,s}}{4s^2}.
\end{align*}
For $I_2$, we use the representation  in \eqref{imcompleteGamma} and the change of variable $t=\frac{t}{2s\log\frac{1}{|x|}}$,
\begin{align*}
I_2&=C_{N,s}\int_0^{\infty} \frac{1}{\Gamma(t)}\int_{|x|^{2s}}^{\infty} \tau^{t-\frac{N}{2s}-1}e^{-\tau}\ d\tau\ dt\\
&=C_{N,s}\int_0^{\infty} \frac{1}{\Gamma(t)}\Gamma(t-\frac{N}{2s},|x|^{2s}) dt\\
&=\frac{C_{N,s}}{|x|^{N}}\int_0^{\infty} \frac{|x|^{2st}}{\Gamma(t)}\int_0^{\infty} \frac{e^{-(1+\tau)|x|^{2s}}}{(1+\tau)^{1-t+\frac{N}{2}}}\ d\tau dt\\
&=\frac{C_{N,s}}{4s^2|x|^{N}(\log\frac{1}{|x|})^2}\int_0^{\infty}\int_0^{\infty} \frac{te^{-t}}{\Gamma(1-\frac{t}{2s\log|x|})} \frac{e^{-(1+\tau)|x|^{2s}}}{(1+\tau)^{1+\frac{\tau}{2s\log|x|}+\frac{N}{2s}}}\  dt d\tau.
\end{align*}
It also follows by the dominated convergent theorem that
\[
\lim_{|x|\to 0}|x|^{N}(\log\frac{1}{|x|})^2I_2=\frac{C_{N,s}}{4s^2}\int_0^{\infty}\int_0^{\infty}te^{-t} \frac{1}{(1+\tau)^{1+\frac{N}{2s}}}\ dtd\tau=\frac{2s}{N}\frac{C_{N,s}}{4s^2}.
\]
This completes the proof of the first part.

To prove the asymptotics at infinity in the second assertion, we start again with the particular case $s=\frac{1}{2}$.  Doing the change of variables, we have
\[
\begin{split}
    G(x)&=\int_0^{\infty} \frac{\gamma_N}{\Gamma(t)}\int_{0}^{\infty}\frac{\tau^{t}e^{-\tau}}{(\tau^2+|x|^2)^{\frac{N+1}{2}}}\ d\tau\ dt\\
    &=\frac{\gamma_N}{|x|^{N+1}}\int_0^{\infty} \frac{1}{\Gamma(t)}\int_{0}^{\infty}\frac{\tau^{t}e^{-\tau}}{\big(\big(\frac{\tau}{|x|}\big)^2+1\big)^{\frac{N+1}{2}}}\ d\tau\ dt\\
    &=\frac{\gamma_N}{|x|^{N}}\int_0^{\infty} \frac{|x|^t}{\Gamma(t)}\int_{0}^{\infty}\frac{\tau^{t}e^{-\tau |x|}}{\big(\tau^2+1\big)^{\frac{N+1}{2}}}\ d\tau\ dt.
\end{split}
\]
Write $|x|^{t}=e^{-{t}\log\frac{1}{|x|}}$. By a double changes of variables, $\mu = {t}{\log\frac{1}{|x|}}$  and $\rho = |x|\tau$, we have
\[
\begin{split}
    G(x)&=\frac{\gamma_N}{|x|^N\log\frac{1}{|x|}}\int_0^{\infty} \frac{e^{-t}}{\Gamma(\frac{t}{\log\frac{1}{|x|}})}\int_{0}^{\infty}\frac{\tau^{-\frac{t}{\log|x|}}e^{-\tau |x|}}{(\tau^2+1)^{\frac{N+1}{2}}}\ d\tau\ dt\\
    &=\frac{\gamma_N}{|x|^{N+1}\big(\log\frac{1}{|x|}\big)^2}\int_0^{\infty} \frac{te^{-t}}{\Gamma(1-\frac{t}{\log|x|})}\int_{0}^{\infty}\frac{\big(\frac{\tau}{|x|}\big)^{-\frac{t}{\log|x|}}e^{-\tau}}{\big(\big(\frac{\tau}{|x|}\big)^2+1\big)^{\frac{N+1}{2}}}\ d\tau\ dt.
\end{split}
\]
It follows by the dominated convergent theorem that
\[
\lim_{|x|\to \infty }\big(\log\frac{1}{|x|}\big)^2|x|^{N+1} G(x) =\gamma_N\int_{0}^{\infty} te^{-t}\int_{0}^{\infty} e^{-\tau}\ d\tau\ dt= \gamma_N.
\]
 For  $s\in (0,1)$, we can directly use the asymptotics at infinity   for $p_s(x,t)$ proved in \cite[Theorem 2.1]{BG60}, namely
 \begin{equation}\label{lim-density}
 \lim_{|x|\to \infty}{|x|^{N+2s}}p_s(x,1)= C_{N,s}=\pi^{-\frac{N}{2}}2^{2s}s\frac{\Gamma(\frac{N}{2}+s)}{\Gamma(1-s)}. 
 \end{equation}
 Hence, by the changing of variables $\tau:=|x|^{2s}\rho$, we have 
 \[
 \begin{split}
  G(x) &=\int_0^{\infty} \frac{1}{\Gamma(t)}\int_{0}^{\infty}p_s(x,\tau)\tau^{t-1}e^{-\tau}\ d\tau\ dt\\
  &=\int_0^{\infty} \frac{|x|^{2ts-2s}}{\Gamma(t)}\int_{0}^{\infty}p_s(x,\tau|x|^{2s})\tau^{t-1}e^{-\tau|x|^{2s}}\ d\tau\ dt.
 \end{split}
 \]
 By the changing of variable $\mu:=t2s\log\frac{1}{|x|}$, we have
\[
\begin{split}
G(x) &=\frac{|x|^{-2s}}{\big(\log\frac{1}{|x|}\big)^2}\int_0^{\infty} \frac{te^{-t}}{\Gamma(1-\frac{t}{2s\log|x|})}\int_{0}^{\infty}p_s(x,\tau|x|^{2s})\tau^{-\frac{t}{2s\log|x|}-1}e^{-\tau|x|^{2s}} d\tau dt\\
&=\frac{1}{\big(\log\frac{1}{|x|}\big)^2}\int_0^{\infty} \frac{te^{-t}}{\Gamma(1-\frac{t}{2s\log|x|})}\int_{0}^{\infty}p_s(x,\tau)\left(\frac{\tau}{|x|^{2s}}\right)^{-\frac{t}{2s\log|x|}}\tau^{-1}e^{-\tau} d\tau dt\\
&=\frac{1}{\big(\log\frac{1}{|x|}\big)^2}\int_0^{\infty} \frac{te^{-t}}{\Gamma(1-\frac{t}{2s\log|x|})}\int_{0}^{\infty}\left( p_s(x\tau^{-\frac{1}{2s}},1)\left(\frac{\tau}{|x|^{2s}}\right)^{-\frac{t}{2s\log|x|}}\right.\times\\
&\qquad\qquad\qquad\qquad\qquad\qquad\qquad\left.\tau^{-\frac{N}{2s}-1}e^{-\tau}\right)\ d\tau dt\\
&=\frac{1}{\big(\log\frac{1}{|x|}\big)^2}\int_0^{\infty} \frac{te^{-t}}{\Gamma(1-\frac{t}{2s\log|x|})}\int_{0}^{\infty}\left(\frac{|x\tau^{-\frac{1}{2s}}|^{N+2s}p_s(x\tau^{-\frac{1}{2s}},1)}{|x|^{N+2s}}\right.\times\\
&\qquad\qquad\qquad\qquad\qquad\qquad\qquad\left.\left(\frac{\tau}{|x|^{2s}}\right)^{-\frac{t}{2s\log|x|}}e^{-\tau}\right) d\tau dt.
\end{split}
\]
Therefore, by \eqref{lim-density},
\[
\lim_{|x|\to \infty}{|x|^{N+2s}\big(\log\frac{1}{|x|}\big)^2}G(x)= C_{N,s}\int_{0}^{\infty}te^{-t}\int_{0}^{\infty} e^{-\tau} \ d\tau dt= C_{N,s}.
\]
The proof of Theorem \ref{Green-1} is completed.
\end{proof}

\section{Functional  setting for the Dirichlet problems}\label{Sect3}

In this section, we set up the functional analytic framework for the study of   Dirichlet problems involving to the operator $\logrels$  in domain.  We aim at investigating properties of bilinear forms and  the operator $\logrels$  from a variational point of view.
Let $\Omega$ be an open set of $\R^N$ and $u,v\in \cC^{\alpha}_c(\R^N)$ understood as functions defined on $\R^N$for some  $\alpha\in (0,1)$. 

Here and the following we identify the space $L^2(\Omega)$ with the space of functions $u\in L^2(\R^N)$  with the property  that  $u\equiv 0$ on $\R^N\setminus\Omega$. For any $s>0$, we consider the bilinear form
 \begin{equation}
 b_{s,\Omega}(u,v) = \frac{1}{2}\int_{\Omega}\int_{\Omega} (u(x)-u(y))(v(x)-v(y))K_s(x-y)\ dxdy,
 \end{equation}
 where we will write $b_s(u,v):= b_{s,\R^N}(u,v)$ if $\Omega=\R^N$. We define the space $H^{\log ,s}(\Omega)$ by
\[
H^{\log ,s}(\Omega):= \left\{ u\in L^2(\Omega):\quad b_{s, \Omega}(u,u)<\infty)\right\},
\]
endowed with the norm
\[
\|u\|_{H^{\log,s}(\Omega)}: \left(\|u\|^2_{L^{2}(\Omega)}+b_{s,\Omega}(u,u)\right)^{\frac{1}{2}},
\]
where  the term $b_{s, \Omega}$ can be written as $b_{s,\Omega}(u,u): =\frac{1}{2}[u]_{s,\log,\Omega}~$  with
\[
 [u]_{s,\log,\Omega}:= \int_{\Omega}\int_{\Omega} |u(x)-u(y)|^2K_s(x-y)\ dxdy,
\]
viewed as  the so-called Gagliardo semi-norm of $u$  for the space  $H^{\log,s}(\Omega)$. We shall also write  $[u]_{s,\log}$  for $[u]_{s,\log,\R^N}$ when $\Omega=\R^N$. To investigate Dirichlet problems subject to boundary (complement) condition, we do need to  introduce the space $ \cH_0^{\log ,s}(\Omega)$ which denotes the closure of the space $\cC_c^{\infty}(\Omega)$ with respect to norm in $H^{\log,s}(\R^N)$,
 \[
 \cH_0^{\log ,s}(\Omega):=\overline{\cC^{\infty}_c(\Omega)}^{\|u\|_{H^{\log,s}(\R^N)}}.
 \]
Note that $\cH^{\log ,s}(\R^N)=\cH_0^{\log ,s}(\R^N)$ and as we shall see below, both $H^{\log ,s}(\Omega)$ and $\cH_0^{\log ,s}(\Omega)$ are Hilbert spaces with respective 
scalar products
\begin{equation}\label{scalar1}
\langle\cdot,\cdot\rangle_{H^{\log,s}(\Omega)}:=\langle\cdot,\cdot\rangle_{L^{2}(\Omega)}+b_{s,\Omega}(\cdot,\cdot)
\end{equation}
and
\begin{equation}\label{scalar2}
\langle\cdot,\cdot\rangle_{\cH_0^{\log,s}(\Omega)}:=\langle\cdot,\cdot\rangle_{L^{2}(\Omega)}+b_{s}(\cdot,\cdot).
\end{equation}
In the next lemma, we present an equivalent representation of the energy $b_s$ associated to the  operator $\logrels$ via Fourier Transform. As we shall see, the will helpful in the sequel as it avoid to always use the asymptotics of the kernel in \eqref{Asymptoti-o}.

\begin{lemma}\label{Energy-Fourier}
Let $u\in \cC^{1}_c(\R^N)$. We have the following equivalent representation of the quadratic  form $b_s$ via Fourier transform
\begin{align*}
\frac{1}{2}\int_{\R^N}\int_{\R^N} |u(x)-u(y)|^2 K_s(x-y)\ dxdy
= \int_{\R^N}\log(1+|\xi|^{2s})|\cF(u)|^2\ d\xi.
\end{align*}
\end{lemma} 
 \begin{proof}
  For fixed $y\in \R^N$ we change coordinates $z = x-y$ and apply Plancherel, recalling
that $\cF(u(\cdot+z))(\xi)=e^{i\xi\cdot z}\cF(u)(\xi)$, we get
\begin{align*}
&\int_{\R^N}\int_{\R^N} |u(x)-u(y)|^2K_s(x-y)\ dxdy\\
&\qquad\qquad\qquad=\int_{\R^N}\int_{\R^N}|e^{i\xi\cdot z}-1|^2|\cF(u)(\xi)|^2\ d\xi K_s(z) dz\\
&\qquad\qquad\qquad=\int_{\R^N}\int_{\R^N}\left(2-2\cos(\xi\cdot z) \right)|\cF(u)(\xi)|^2\ d\xi K_s(z) dz\\
&\qquad\qquad\qquad=\int_{\R^N}\int_{\R^N}\left(2-e^{i\xi\cdot z}-e^{-i\xi\cdot z}\right)|\cF(u)(\xi)|^2\ d\xi K_s(z) dz.
\end{align*}
We next use \eqref{KerenelExp} and the radial property of the  transition density functions $p_s(z,t)$ to get
\begin{align*}
&\int_{\R^N}\int_{\R^N}\big(2-e^{i\xi\cdot z}-e^{-i\xi\cdot z}\big)|\cF(u)(\xi)|^2  \ d\xi K_s(z) dz\\
&\qquad\qquad\qquad\qquad=2\int_{\R^N}\int_{0}^{\infty}\frac{e^{-t}}{t}\big(1-e^{-t|\xi|^{2s}}\big)\ dt|\cF(u)(\xi)|^2 d\xi \\
&\qquad\qquad\qquad\qquad=2\int_{\R^N}\log(1+|\xi|^{2s})|\cF(u)(\xi)|^2 d\xi. 
\end{align*}
This provides the proof  of Lemma \eqref{Energy-Fourier}.
 \end{proof}

\begin{lemma}
\label{Properties}
Let $\Omega$ be an open subset of $\R^N$ and $s\in (0,1]$.
\begin{itemize}
\item[(i)] The spaces $H^{\log,s}(\Omega)$ and $\cH_0^{\log,s}(\Omega)$  are a Hilbert spaces with respective scalar products given in \eqref{scalar1} and \eqref{scalar2}. 
\item[(ii)] If $u\in H^{\log,s}(\Omega)$, then  $|u|,u^{\pm}\in H^{\log,s}(\Omega)$ with $$\||u|\|_{H^{\log,s}(\Omega)},\|u^{\pm}\|_{H^{\log,s}(\Omega)}\le \|u\|_{H^{\log,s}}.$$
\item[(iii)]  The space $\cC_c^{0,\alpha}(\R^N)\subset  H^{\log,s}(\R^N)$ for any $\alpha>0$.
\item[(iv)] \textit{Multiplication of a function $u\in H^{\log}(\R^N)$ and a function  $\phi\in \cC^{0,\alpha}(\R^N)$}.\\ Let  $u\in H^{\log,s}(\R^N)$ and $\phi\in \cC^{0,\alpha}(\R^N)$  , then $\phi u\in H^{\log,s}(\R^N)$ and there is  a constant $C:=C(N, s, \alpha, \phi)>0$ such that 
 \[
 \|\phi u\|_{H^{\log,s}(\R^N)}\le C\|u\|_{H^{\log,s}(\R^N)}.\qquad 
 \]
The same result  easily holds for  $\phi\in \cC^{0,\alpha}(\R^N)$  with  compact support and $u\in H^{\log,s}(\R^N)$.
\item[(v)] 
The space  $H^{\log,s}(\R^N)$ is larger than any classical Sobolev space. In fact, we have that $H^{m}(\R^N)\subset H^{\log,s}(\R^N)$ for any $m\in\N$, where $H^{m}(\R^N)=W^{m,2}(\R^N)$ is the classical Sobolev  space with the  norm given by
\begin{equation}
\|u\|_{H^m(\R^N)}:=\sum_{0\le |\alpha|\le m}\|D^{\alpha}u\|_{L^2(\R^N)}\quad
\end{equation}
\text{with}
\[
\alpha=(\alpha_1,\cdots,\alpha_N)\qquad\text{ and }~\quad |\alpha|=\sum_{j=1}^N\alpha_j.
\]
 \end{itemize}
\end{lemma}	

\begin{proof} $(i)$. We provide the proof    only for the space $H^{\log,s}(\Omega)$. The proof for $\cH^{\log,s}(\Omega)$ follows analogously.

Let $\{u_n\}_n\subset H^{\log,s}(\Omega)$ be a Cauchy sequence. Then  $\{u_n\}_n$ is in particular a Cauchy sequence in $L^2(\Omega) $ and
hence there exists a  $u\in L^2(\Omega)$ such  that $u_n\to u$ as $n\to \infty$. Passing to a subsequence if necessary, we get that $u_n\to u$ a.e in $\R^N$ as $n\to \infty$ and by Fatou Lemma we have
\[
b_{s,\Omega}(u,u)\le \liminf_{n\to \infty}b_{s,\Omega}(u_n,u_n)\le \sup_{n\in \N}b_{s,\Omega}(u_n,u_n)<\infty,
\]
showing that $u\in H^{\log,s}(\Omega)$. Apply once more Fatou Lemma it follows that
\begin{align*}
\|u_n-u\|^2_{H^{\log,s}(\Omega)}&=\|u_n-u\|^2_{L^2(\Omega)}+ b_{s,\Omega}(u_n-u,u_n-u)\\
&\le  \liminf_{n\to \infty}\|u_n-u_m\|^2_{H^{\log,s}(\Omega)},
\end{align*}
 for $ n,m\in \N$. The claim follows since $\{u_n\}_n$ is a Cauchy sequence in $H^{\log,s}(\Omega)$.
 
$(ii)$. It straightforward to see by integrating  the inequality $$|u(x)|-|u(y)|\le |u(x)-u(y)|$$
that 
$$\text{$b_{s,\Omega}(|u|,|u|)\le b_{s,\Omega}(u,u,)$\quad and\quad  $\||u|\|_{H^{\log,s}(\Omega)}\le \|u\|_{H^{\log,s}(\Omega)}$.}$$ 
 Next, using also the inequality 
 $$2(u^+(x)-u^+(y))(u^{-}(x)-u^{-}(y))=- 2(u^{-}(x)u^{+}(y)+u^{-}(y)u^{+}(x))\le 0 \text{ for } \ x,y\in \R^N,$$
  it follows   by a simple computation that
 \[
 \begin{split}
 b_{s,\Omega}(u,u)&= b_{s,\Omega}(u^{+},u^{+})+b_{s,\Omega}(u^{-},u^{-}) -2b_{s,\Omega}(u^{+},u^{-})\\
 &\ge b_{s,\Omega}(u^{+},u^{+})+b_{s,\Omega}(u^{-},u^{-}),
  \end{split}
 \]
 proving clearly that   $(ii)$  holds.  
 
 $(iii)$. Let $u\in \cC_c^{0,\alpha}(\R^N)$ be such that $\supp \ u \subset B_r$, $r>0$. without loss of generality we may assume that $r=1$ such that we can directly apply the asymptotics in \eqref{Asymptoti-o}. It clear that $u\in \cC_c^{0,\alpha}(\R^N)$ implies $u\in L^2(\R^N)$. Moreover, for $x,y\in \R^N$,  fix $A:= \|u\|_{\cC^{0,\alpha}(\Omega)}$, we have
 \begin{equation}\label{min-iq}
 |u(x)-u(y)|\le 2\|u\|_{L^{\infty}}\qquad\text{and}\quad  |u(x)-u(y)|\le A|x-y|^{\alpha}
  \end{equation}
 Hence 
 \[
  |u(x)-u(y)|\le  C\min\{1,|x-y|^{\alpha}\}.
 \]
 Therefore,
 \begin{align*}
  b_{s}(u,u)&=\frac{A}{2}\int_{B_1}\int_{B_1}|x-y|^{2\alpha} K_s(x-y) \ dx dy   +A\int_{B_1}\int_{\R^N\setminus B_1}K_s(x-y)\ dydx\\
  &\le C_1\int_{B_1}\int_{B_1}|x-y|^{2\alpha-N}   dx dy +C_2\int_{B_1}\int_{\R^N\setminus B_1}|x-y|^{-N-2s}\ dy\\
  &\le C_3 |B_1(0)|\big(\int_{B_1}|z|^{2\alpha-N} \ dz +\int_{\R^N\setminus B_1}|z|^{-N-2s}\ dz\big)\\
  &= C_3 \omega_{N-1} |B_1(0)|\big( \int_{0}^1\rho^{2\alpha-1} \ d\rho +\int_{1}^{\infty}\rho^{-1-2s}\ d\rho\big) = C<\infty.
 \end{align*}
 where the constant $C:=C(N,s,\alpha)>0$. This ends the proof for $(iii)$ since  $u\in L^2(\R^N)$.  
 
 $(iv)$. Let $u\in H^{\log,s}(\R^N)$ and $\phi\in \cC^{0,\alpha}(\R^N)$.   Then, $\phi u\in L^{2}(\R^N)$ since $\phi$ is bounded.  Next,  we the constant $A$ as in \eqref{min-iq}, we have the following   inequality 
\begin{align*}
 |\phi(x)u(x)-\phi(y)u(y)|^2&\le 2\Big(|u(x)-u(y)|^2|\phi(x)|^2+|u(y)|^2|\phi(x)-\phi(y)|^2\Big)\\
 &\le 2\Big(|u(x)-u(y)|^2|\phi(x)|^2+ A^2|u(y)|^2\min\{1,|x-y|^{2\alpha}\}\Big).
 \end{align*} 
  It follows that
 \begin{align*}
 b_{s}(\phi u,\phi u)&\le \int_{\R^N}\int_{\R^N}|\phi(x)|^2|u(x)-u(y)|^2K_s(x,y)  dx dy \\
   &\hspace{2cm} +2A^2\int_{\R^N}\int_{ \R^N}|u(x)|^2\min\{1,|x-y|^{2\alpha} K_s(x-y) dydx\\
   &= I_1+I_2.
 \end{align*}
Since $\phi$ is bounded, $I_1$ is finite with $I_1\le  2\|\phi\|^2_{L^{\infty}(\R^N)} b_{s}(u,u)$. Moreover, using \eqref{Asymptoti-o}, the quantity $I_2$ can be bound as follows
\begin{align*}
I_2&\le 2A^2\int_{B_1}\int_{ B_1}|u(x)|^2|x-y|^{2\alpha-N} dydx +2A^2\int_{B_1}\int_{ \R^N\setminus B_1}\frac{|u(x)|^2}{|x-y|^{N+2s}} dydx\\
&\le C\|u\|_{L^2(\R^N)}\Big(\int_{ B_1}|z|^{2\alpha-N} dz+\int_{ \R^N\setminus B_1}|z|^{-2s-N} dz\Big)\\
&=C\omega_{N-1}\|u\|_{L^2(\R^N)}\Big(\int_{ 0}^1\rho^{2\alpha-1} d\rho+\int_{ 1}^{\infty}\rho^{-2s-1} d\rho\Big):= C_1 \|u\|_{L^2(\R^N)}.
\end{align*}
It follows combining $I_1$ and $I_2$ that there is a constant $C:=C(N,s,\alpha,\phi)$ such that
\[
 \|\phi u\|_{H^{\log,s}(\R^N)}\le C\|u\|_{H^{\log,s}(\R^N)}.\qquad 
 \]
This completes the proof for $(iv)$. 

$(v).~ $We use Lemma \ref{Energy-Fourier}. Put $m=s\beta>0$ for some  $\beta>0$. Using the trivial inequality (see \cite{JSW20}) ~ $\log a\le \frac{a^{\beta}}{\beta}$, $a\ge 1$,  we get
\begin{align*}
\|u\|^2_{H^{s,\log}(\R^N)}&=\int_{\R^N}\left(1+\log(1+|\xi|^{2s})\right)|\cF( u)(\xi)|^2\ d\xi\\
&\le \int_{\R^N}\left(1+\frac{(1+|\xi|^{2s})^\beta}{\beta}\right)|\cF( u)(\xi)|^2\ d\xi\\
&\le \int_{\R^N}\left(1+\frac{2^{\beta-1}(1+|\xi|^{2s\beta})}{\beta}\right)|\cF( u)(\xi)|^2\ d\xi\\
&\le C_{s,\beta}\int_{\R^N}\left(1+|\xi|^{2m}\right)|\cF( u)(\xi)|^2\ d\xi= C_{\beta,s}\|u\|_{H^{m}(\R^N)}.
\end{align*}

This provides the proof for $(v)$ and thus the proof of Lemma \ref{Properties}.
\end{proof}

\begin{lemma}\label{Density}
\begin{enumerate}[(i)]
    \item The space $\cC_c^{\infty}(\R^N)$ is dense in  $H^{\log,s}(\R^N)$, that is if $u\in H^{\log,s}(\R^N)$, then there exists a sequence $(u_n)_{n\in\N}\subset \cC_c^{\infty}(\R^N)$ such that 
\[
\|u_n-u\|_{H^{\log,s}(\R^N)}\to 0\quad \text{ as }\quad n\to \infty.
\]
\item  Then  the space $\cC^{\infty}_c(\Omega)$ is dense in $H^{\log,s}(\Omega)$ for any $s\in (0,1]$. 
 \end{enumerate}
\end{lemma}

\begin{proof}
The proof of  Lemma \ref{Density} can be deduced from \cite{FJ22}, where general nonlocal operators of \textit{small order} is considered, combining Theorem 3.1, Lemma 3.3 and Proposition 3.4 therein. We provide a new idea of $(i)$  following the roadmap of  \cite[Theorem 7.14]{LL97}.
This uses the fact that the space $H^{\log,s}(\R^N)$ is larger than any Sobolev space $H^{m}(\R^N)$ of order  $m\ge 0$ with norm 
$$\|u\|_{H^{m}(\R^N)}:=\int_{\R^N}\left(1+|\xi|^{2m}\right)|\cF( u)(\xi)|^2\ d\xi.$$ 
In fact, it follows from Theorem \ref{properties1} $(iii)$ that the space $H^{\log,s}(\R^N)$  can be equivalently  defined in term of Fourier transform, 
$$H^{\log,s}(\R^N):=\big\{u\in L^2(\R^N): \int_{\R^N}\log(1+|\xi|^{2s})|\cF( u)(\xi)|^2\ d\xi<\infty\big\}$$
with the corresponding norm
\begin{equation}\label{Norm-Fourier}
\|u\|_{H^{\log,s}(\R^N)}=\Big(\int_{\R^N}\big(1+\log(1+|\xi|^{2s})\big)|\cF( u)(\xi)|^2\ d\xi\Big)^{\frac{1}{2}}.
\end{equation}
Therefore, if $u\in H^{m}(\R^N)$, it follows from Theorem \ref{properties1} $(v)$ that $u\in H^{\log,s}(\R^N)$.
Therefore the embedding  $H^{m}(\R^N)\subset H^{\log,s}(\R^N)$  is continuous.  To conclude the claim, it  suffices to show  that $H^{m}(\R^N)\subset H^{\log,s }(\R^N)$  densely.  For   $u\in H^{\log,s}(\R^N)$,   take 
$$\hat u_n(\xi)=\big(1+\frac{|\xi|^{2m}}{n}\big)^{-1}\cF( u)(\xi).$$ 
Then the sequence $\{u_n\}_{n\ge 1}\subset H^{m}(\R^N)$. Indeed, first observe  from Theorem \ref{properties1} $(iv)$  that 
\begin{equation}
    \log(1+|\xi|^{2s})\cF( u_n)= \big(1+\frac{|\xi|^{2m}}{n}\big)^{-1}\log(1+|\xi|^{2s})\cF( u).
\end{equation}
 Moreover, since $n\ge 1$, We have with $R_0>0$,
\begin{align*}
&\|u_n\|_{H^{m}(\R^N)}=\int_{\R^N}(1+|\xi|^{2m}|)\cF( u_n)(\xi)|^2 d\xi\\
&\qquad=\int_{\R^N}(1+|\xi|^{2m}|)\big(1+\frac{|\xi|^{2m}}{n}\big)^{-1}|\cF( u)(\xi)|^2 d\xi\\
&\qquad\le n\int_{B_{R_0}}|\cF( u)(\xi)|^2 d\xi+\frac{n}{\log(1+R_0^{2s})}\int_{B^c_{R_0}}\log(1+|\xi|^{2s})|\cF( u)(\xi)|^2 d\xi\\
&\qquad\le C_{n,R_0,s}\int_{\R^N}\big(1+\log(1+|\xi|^{2s})\big)|\cF( u)(\xi)|^2\ d\xi\le C\|u\|_{H^{\log,s}(\R^N)}<\infty.
\end{align*}
It thus follows by the dominated convergence theorem that
\[
\|u_n-u\|^2_{H^{\log,s}(\R^N)}=\int_{\R^N}\big(1+\log(1+|\xi|^{2s})\big)\Big|1-\big(1+\frac{|\xi|^{2m}}{n}\big)^{-1}\Big|^2|\cF( u)(\xi)|^2\ d\xi\to 0
\]
$\text{as }\ n\to \infty$. This completes the proof of Lemma \ref{Density}.
\end{proof}

We have the following Poincar\'e inequality and embedding results for the  fractional logarithmic Schr\"odinger operator $\logrels$ for $s\in (0,1]$.

\begin{prop}[\textbf{Poincar\'e inequality}]\label{Poincare} Let $\Omega\subset \R^N$ and $u\in \cH^{\log,s}_0(\Omega)$.
\begin{enumerate}[(i)]
\item If $|\Omega|<\infty$. Then, there exists a constant $C:=C(\Omega,N,s)>0$ such that
\begin{equation} \label{Poin}
\|u\|_{L^2(\Omega)}\le C\int_{\R^N}\int_{\R^N}(u(x)-u(y))^2K_s(x-y)\ dxdy.
\end{equation}
\item If $\Omega:=(-a,a)\times\R^{N-1})$, $a>0$, is bounded in one direction, that is, $|x_1|\le a$, then,
\begin{equation}
\|u\|_{L^2((-a,a)\times\R^{N-1})}\le C\int_{\R^N}\int_{\R^N}(u(x)-u(y))^2K_s(x-y)\ dxdy.
\end{equation}
\item Let $p\in [1,2]$. The embedding  $H^{\log,s}(\R^N)\hookrightarrow  L^p(\R^N)$ is locally compact. In particular, the embedding  $\cH_0^{\log,s}(\Omega)\hookrightarrow  L^p(\Omega)$ is compact for any open set $\Omega\subset \R^N$ with $|\Omega|<\infty$.  
\end{enumerate}
\end{prop}

\begin{proof} $(i)$
The  Poincarr\'e inequality  can be deduced from \cite{F22,FKV15,JW20, FJ22} but we include the proof here for completeness.
Set 
\[
\Xi(R,\Omega):= \frac{2}{\log(1+R^{2s})(1- (2\pi)^{-N}R^{N}|\Omega||B_1(0)|)}.
\]
  Since $u= 0$ in $\R^N\setminus \Omega$,  we first have by  H\"older  inequality that 
\[
|\cF( u)(\xi)|^2\le (2\pi)^{-N}|\Omega|\|u\|^2_{L^2(\Omega)}\qquad\text{ for every }\quad \xi\in \R^N.
\]
Next, by Plancherel theorem  and for any $R>0$,  we get
 \begin{align*}
&\|u\|^2_{L^2(\Omega)} = \int_{\R^N}|\cF( u)(\xi)|^2\ d\xi \\
&\quad= \int_{|\xi|<R}|\hat u(\xi)|^2\ d\xi+ \int_{|\xi|\ge R}\log(1+|\xi|^{2s})|\cF( u)(\xi)|^2\log(1+|\xi|^{2s})^{-1} d\xi\\
&\quad\le \frac{R^{N}|\Omega||B_1(0)| \|u\|^2_{L^2(\R^N)}}{(2\pi)^{N}}+\frac{1}{2\log(1+R^{2s})}\int_{\R^N}\log(1+|\xi|^{2s})|\cF( u)(\xi)|^2 d\xi.
\end{align*}
Therefore, choose $R<2\pi (|\Omega||B_1(0)|)^{-\frac{1}{N}}=  2\pi\big(\frac{N}{\omega_{N-1}|\Omega|}\big)^{\frac{1}{N}}$ we find using Lemma \ref{Energy-Fourier} that 
\begin{align*}
\|u\|^2_{L^2(\Omega)}&\le \Xi(\Omega,R)\int_{\R^N}\int_{\R^N}(u(x)-u(y))^2K_s(x-y)\ dxdy.
\end{align*}
The  proof of $(i)$ follows by minimizing the quantity  $\Xi(\Omega,R)$ with respect to $R$ in  the above inequality.  

$(ii).$ This can be deduced from \cite{JW20}. Indeed,  if $\Omega=(-a,a)\times\R^{N-1}$ for some $a>0$, we write $x=(x_1,x')\in(-a,a)\times\R^{N-1}$. By Fubini's theorem and the change of variable $\zeta:= \frac{x'-y'}{|x_1-y_1|}$, we can first observe that
\[
\begin{split}
    \kappa_{N,a,s}(x)&:= \int_{\R\setminus(-a,a)}\int_{\R^{N-1}}\frac{1}{|x-y|^{N+2s}}\ dy'dy_1\\
    &= \int_{\R\setminus(-a,a)}|x_1-y_1|^{N-2s}\int_{\R^{N-1}}\Big(1+\Big|\frac{x'-y'}{|x_1-y_1|}\Big|^2\Big)^{-\frac{N+2s}{2}}\ dy'dy_1\\
    &= \int_{\R\setminus(-a,a)}|x_1-y_1|^{-1-2s}\int_{\R^{N-1}}\big(1+|\zeta|^2\big)^{-\frac{N+2s}{2}}\ d\zeta dy_1\\
    &= \omega_{N-2}\Big(\int_{-\infty}^{-a}(x_1-y_1)^{-1-2s} dy_1+\int_{a}^{\infty}(y_1-x_1)^{-1-2s} dy_1\Big)\times\\
    &\qquad\qquad\qquad\qquad\int_{0}^{\infty}\frac{\rho^{N-2}}{(1+\rho^2)^{\frac{N+2s}{2}}}\ d\rho \\
    &=\frac{\pi^{\frac{N-1}{2}}\Gamma(\frac{1+2s}{2})}{2s\Gamma(\frac{N+2s}{2})}\big((a+x_1)^{-2s}+(a-x_1)^{-2s}\big):= M(N,s,a).
\end{split}
\]
Therefore, using \eqref{Asymptoti-o} in Theorem \ref{properties1}, we can find a constant $C':=C'(N,s,\Omega)>0$ such that 
\[
\begin{split}
    \int_{\R^N}\int_{\R^N}(u(x)-u(y))^2K_s(x-y)\ dxdy&\ge 2C'\int_{\Omega}|u(x)|^2\kappa_{N,a,s}(x)\ dx\\
    &\ge C'M(N,s,a)\|u\|^2_{L^2(\Omega)}=C\|u\|^2_{L^2(\Omega)}.
\end{split}
\]
$(iii).$ This is just an application of  \cite[Theorem 1.1 and 1.2]{JW20}, we give a detailed proof this operator. Let $w\in L^1(\R^N)$. Then the convolution operator
\begin{equation}\label{comp}
T_{w}: L^2(\R^N)\to L^2(\R^N), \qquad T_w = w*u 
\end{equation}
is continuous and  locally compact (see \cite[Lemma 2.1]{JW20}, see also \cite[Theorem 3.81]{GF20}).  Moreover, let $\delta>0$ be a small parameter. Then, from Theorem \ref{properties1}, we have 
\[
\begin{split}
\|K_s^{\delta}\|_{L^1(\R^N)}:= \int_{|z|\ge \delta}K_s(z)\ dz& = \int_{1\ge |z|\ge \delta}K_s(z)\  dz+ \int_{ |z|\ge 1}K_s(z)\  dz\\
&=C(N,s)(\log\delta+\frac{1}{2s})<\infty.
\end{split}
\]
Therefore, let $\{\delta_n\}_n$ be a sequence such that $\delta_n\to 0$ as $n\to \infty$. Set 
\[
w_n:=w_{\delta_n}= {K_s^{\delta_n}}\big/{\|K_s^{\delta_n}\|_{L^1(\R^N)}}
\]
Thus $\|w_n\|_{L^1(\R^N)}=1$. Then, by the evenness  of $w_n$ and the Jensen's inequality, we have  
\begin{equation}\label{pre-comp}
\begin{split}
   \|u-T_{w_n}\|_{L^2(\R^N)}&= \int_{\R^N}\big(u(x)-T_{w_n}u(x)\big)^2\ dx\\
   &=\int_{\R^N}\Big(\int_{\R^N}[u(x)-u(x+y)]w_n(y)\ dy\Big)^2\ dx\\
   &\le \int_{\R^N}\int_{\R^N}|u(x)-u(x+y)|^2w_n(y)\ dy\ dx\\
   &\le \|w_n\|^{-1}_{L^1(\R^N)}\int_{\R^N}\int_{\R^N}|u(x)-u(x+y)|^2K_s(y)\ dy\ dx\\
   &\le \|w_n\|^{-1}_{L^1(\R^N)}\|u\|_{H^{\log,s}(\R^N)}.
\end{split}
\end{equation}
Now, let $R_KT=1_{K}T$ be the multiplication of an operator $T$ with the  characteristic
function  $1_{K}$ of $K$,  where $K$ is a compact subset of $\R^N$. Then, by \eqref{pre-comp},
\[
\begin{split}
\|R_K-R_KT_{w_n}\|_{\cL(H^{\log,s}(\R^N),L^2(K))}&=\sup_{ \|u\|_{H^{\log,s}(\R^N)}\le 1}\|R_Ku-R_KT_{w_n}u\|_{L^{2}(K)}\\
&\le  \|w_n\|^{-1}_{L^1(\R^N)}\to 0\qquad \text{ as } n\to 0.
\end{split}
\]
Since $R_KT_{w_n}$ is  compact for every $n\in \N$, it follows that  
$
R_KT_w :H^{\log,s}(\R^N) \to L^2(K)
$
is  compact. Therefore,   the embedding  $H^{\log,s}(\R^N)\hookrightarrow  L^2(\R^N)$ is locally compact. Furthermore,  the embedding  $\cH_0^{\log,s}(\Omega)\hookrightarrow  L^2(\Omega)$ with $|\Omega|<\infty$ is also compact    thanks to the poincar\'e inequality in \eqref{Poin} which provides the continuous embedding. Since also $L^2(\Omega)\hookrightarrow  L^p(\Omega)$ for $p\in [1,2]$ and $|\Omega|<\infty$ is  continuous, it follows also  that $\cH_0^{\log,s}(\Omega)\hookrightarrow  L^p(\Omega)$ for $p\in [1,2]$ is compact. This completes the proof of the Proposition \ref{Poincare}.
 \end{proof}

As consequence of the Poincar\'e inequality, we have for  bounded set $\Omega\subset \R^N$  with continuous boundary that  the space $\cH_0^{\log,s}(\Omega)$ can be identified by
 \[
\cH_0^{\log,s}(\Omega)=\Big \{ u\in H^{\log,s}(\R^N): \quad u\equiv 0 \text{ \ on \ }\ \R^N\setminus \Omega\Big\}
\]
 and it is a Hilbert  space endowed with the scalar product
$
	(v,w) \mapsto b_{s}(v,w) 
$
and the corresponding norm ~$\|u\|_{\H_0^{\log,s}(\Omega)}=\sqrt{b_{s}(u,u)}$.
 
\begin{prop}[{\bf Poincar\'e Wirtinger}] \label{P-W} Assume that $\Omega\subset \R^N$ is an open bounded set. Then  there exists a constant $C>0$, depending only on $\Omega$ and $s$, such that for every function $u\in H^{\log,s}(\Omega)$ 
    \begin{equation}\label{Poincare-Wirtinger}
    \int_{\Omega}\big|u-\frac{1}{|\Omega|}\int_{\Omega}u \big|^2\ dx\le C\int_{\Omega}\int_{\Omega}(u(x)-u(y))^2K_{s}(x-y)\ dxdy   
    \end{equation}
\end{prop}
\begin{proof}
   We argue by contradiction. Assume that the inequality in \eqref{Poincare-Wirtinger} does not hold. Then, there exists a sequence of functions $\{u_n\}_{n\ge 1}\subset H^{\log,s}(\Omega)$ such that 
\[
\frac{1}{|\Omega|}\int_{\Omega} u_n\ dx = 0,\quad\quad \int_{\Omega}|u_n|^2 \ dx =1
\]
and 
\begin{equation}\label{contrad}
\int_{\Omega}\int_{\Omega}(u_n(x)-u_n(y))^2K_{s}(x-y)\ dxdy<\frac{1}{n}.
\end{equation}
Then $\{u_n\}_{n\ge 1}$ is uniformly bounded in $ H^{\log,s}(\Omega)$. Passing to a subsequence, we have 
\[
u_n\to u \quad \text{ in }\quad L^2(\Omega)
\]
thanks to the embedding in $(iii) $ of proposition \ref{Poincare}. It thus follows that 
\begin{equation}\label{Normalized}
\frac{1}{|\Omega|}\int_{\Omega} u\ dx = 0,\quad \text{ and } \quad \int_{\Omega}|u|^2 \ dx =1.
\end{equation}
On the other hand, passing to the limit in \eqref{contrad} implies that
\[
\int_{\Omega}\int_{\Omega}(u(x)-u(y))^2K_{s}(x-y)\ dxdy=0.
\]
Therefore, $u$ is constant  in $\Omega$ by \cite[Propistion 2.10]{FJ22}, which contradicts \eqref{Normalized}. This completes the proof of Proposition \ref{P-W}.
\end{proof}

\begin{lemma}\label{s-embeding}
Let $0<s< s_0$ and $u:\R^N\to \R$ be a measurable function. Then  there exists a positive constant $C:=C(N,s_0)$ such that  
\[
\|u\|_{H^{\log, s}(\R^N)}\le C\|u\|_{H^{\log, s_0}(\R^N)}.
\]
In particular, ~$H^{\log, s_0}(\R^N)\subset H^{\log, s}(\R^N)$.
\end{lemma}

\begin{proof}
By the Poincar\'e inequality, there exists  $M:=M(N,s_0)>0$ such that 
\[
\|u\|^2_{L^2(B_1(0))}\le M\int_{\R^N}\int_{\R^N}(u(x)-u(y))^2K_{s_0}(x-y)\ dxdy 
\]
 for all  $ u\in H^{\log, s_0}(\R^N)$, where the constant $M$ is given by
\[
M=\min_{R\in \big(0,\frac{2\pi N}{\omega_{N-1}|\Omega|}\big)}\Big\{\frac{2}{\log(1+R^{2s_0})\Big(1- (2\pi)^{-N}R^{N}|B_1(0)|^2\Big)}\Big\}.
\]
Moreover,  by Plancherel theorem,  we get with $C_1:=( M\log(2)+1)$,
\begin{align*}
&\int_{\R^N}\int_{\R^N}|u(x)-u(y)|^2K_s(x-y)\ dxdy =\int_{\R^N}\log(1+|\xi|^{2s})|\cF( u)(\xi)|^2\ d\xi\\
&\hspace{2cm}=\Big( \int_{|\xi|<1}+\int_{|\xi|\ge1}\Big)\log(1+|\xi|^{2s})|\cF( u)(\xi)|^2\ d\xi\\
&\hspace{2cm}\le \log(2)\|u\|^2_{L^2(B_1)}+\int_{\R^N}\log(1+|\xi|^{2s_0})|\cF( u)(\xi)|^2\ d\xi\\
&\hspace{2cm}\le (\log(2)M+1)\int_{\R^N}\log(1+|\xi|^{2s_0})|\cF( u)(\xi)|^2\ d\xi\\
&\hspace{2cm}=C_1\int_{\R^N}\int_{\R^N}|u(x)-u(y)|^2K_{s_0}(x-y)\ dxdy.
\end{align*}
Therefore, it  follows that 
\begin{align*}
\|u\|^2_{H^{\log, s}(\R^N)}&= \|u\|^2_{L^2(\R^N)}+\int_{\R^N}\int_{\R^N}|u(x)-u(y)|^2K_{s}(x-y)\ dxdy\\
&\le \|u\|^2_{L^2(\R^N)}+C_1\int_{\R^N}\int_{\R^N}|u(x)-u(y)|^2K_{s_0}(x-y)\ dxdy\\
&= C\|u\|^2_{H^{\log,s_0}(\R^N)},
\end{align*}
where $C:=\max\{1,C_1\}$. This completes the proof of Lemma \ref{s-embeding}.
\end{proof}

\begin{thm}[{\bf Logarithmic inequality}] \label{logarithmic-ineq}
Let  $u\in H^{\log,s}(\R^N)$ be any function and let $s>0$ be any real number. Then,
\begin{align*}
\int_{\R^N}|u|^2\log|u|^2\ dx\le B_{N}\|u\|^2_{L^2(\R^N)}+ \|u\|^2_{L^2(\R^N)} \log\big(\|u\|^2_{L^2(\R^N)} \big)+\frac{N}{2s}b_{s}(u,u),
\end{align*}
where the constant  $B_{N}$ is given
\begin{equation}\label{c1}
 B_{N}:= \log\Big(\frac{\Gamma(N)}{\Gamma\Big(\frac{N}{2}\Big)}\Big)-\frac{N}{2}\Big(\log(4\pi)+2\psi\Big(\frac{N}{2}\Big)\Big).
\end{equation}
$\text{with }\ \psi(t)=\log(\Gamma(t))'$,  the digamma function. 

For general open set $\Omega\subset \R^N$, we have
\begin{equation}\label{Log-Ineq-O}
\begin{split}
\int_{\Omega}|u(x)|^2\log\Big(\frac{|u(x)|^2}{~\|u\|^2_{L^2(\Omega)}}\Big)\ dx
\le \frac{N}{2s}\big(B_{N}\|u\|^2_{L^2(\Omega)} +b_{\Omega,s}(u,u)\big).
\end{split}
\end{equation}
\end{thm}

\begin{proof}
 Let us first recall the following logarithmic inequality due to Beckner \cite[Theorem 3]{B95}.  With  $L^2$-normalization $\|u\|_{L^2(\R^N)}=1$, it holds that
\begin{equation}\label{Beckner}
\int_{\R^N}|u|^2\log|u|^2\ dx\le {N}\int_{\R^N}\log|\xi||\cF(u)(\xi)|^2\ d\xi+2B_N ,
\end{equation}
where the constant $B_N$ is given by \eqref{c1}.
Up to conformal automorphism, extremal functions for \eqref{Beckner} are functions of the form $$ u(x)=A_N(1+|x|^2)^{-\frac{N}{2}}\qquad \text{  with \ }\ A_N= \frac{2^N\Gamma(\frac{N+1}{2})}{\sqrt{\pi}\Gamma(\frac{N}{2})}.$$
From the inequality $\log|\xi|=\frac{1}{2s}\log|\xi|^{2s}\le \frac{1}{2s}\log(1+|\xi|^{2s})$ and assuming the $L^2$-normalization $\|u\|_{L^2(\R^N)}=1$ , if we multiply   \eqref{Beckner}   by the factor $2s$, we get 
\begin{equation}\label{Beckner-2}
\int_{\R^N}|u|^2\log|u|^2\ dx\le \frac{N}{2s}b_s(u,u)+2B_N. 
\end{equation}
Next, set $u=\frac{v}{\|v\|^2_{L^2}}$, the complete inequality is given by
\begin{align*}
\int_{\R^N}|v|^2\log|v|^2\ dx\le \frac{N}{2s}b_s(v,v)+2B_N\|v\|_{L^2(\R^N)}+\|v\|^2_{L^2(\R^N)} \log\big(\|v\|^2_{L^2(\R^N)}).
\end{align*}   
This completes the proof of Theorem \ref{logarithmic-ineq}. 
\end{proof}

 An example of inequality in \eqref{Log-Ineq-O} is the Beckner's logarithmic Sobolev inequality on $\mathbb{S}^N$ \cite{B97, FKT20} with $\Omega$ being the unit sphere in $\R^{N+1}$.  Up to conformal automorphism, extremal functions for \eqref{Log-Ineq-O} with $\Omega\equiv \mathbb{S}^N$ are functions of the form
\[
u(\omega)= c\left(\frac{\sqrt{1-|\xi|^2}}{1-\xi\cdot\omega}\right)^{N/2}
\]
for some $\xi\in\R^{N+1}$ with $|\xi|<1$ and some $c\in\R$.

\section{Dirichlet problems}\label{Sect4}
We have so far introduced sufficient tools to characterized the operator $(I+(-\Delta)^s)^{\log}$ in weak sense.
Let $\Omega\subset \R^N$ be an open bounded set  and $f\in L^2(\Omega)$. Consider the Poisson problem,
\begin{equation}\label{eq-omega}
	\begin{split}
	\quad\left\{\begin{aligned}
		(I+(-\Delta)^s)^{\log}  u &= f &&\text{ in\ \quad  $\Omega$}\\
		u &=  0             && \text{ on\ \  } \R^N\setminus\Omega. 
	\end{aligned}\right.
	\end{split}
	\end{equation}
 We recall that the belinear form $b_s$ is defined by
 \[
 b_s(u,v) =\frac{1}{2}\int_{\R^N}\int_{\R^N}(u(x)-u(x))(v(x)-v(y))K_s(x-y)\ dxdy.
 \]
We call $u\in \cH_0^{\log,s}(\Omega)$ a weak solution to problem \eqref{eq-omega} if
\[
b_s(u,\phi)=\int_{\Omega}f\phi\ dx,\quad \text{ for all }\quad\phi\in \cH_0^{\log, s}(\Omega).
\]
By Cauchy Schwartz and Poincar\'e inequality, we have
\[
|b_s(u,\phi)|\le \sqrt{b_s(u,u)}\sqrt{b_s(\phi,\phi)}=\|u\|_{\cH^{\log,s}(\Omega)}\|\phi\|_{\cH^{\log,s}(\Omega)}, 
\]
and
$$
 \big|\int_{\Omega}f\phi\ dx\big|\le C\|\phi\|_{\cH^{\log,s}(\Omega)}.
$$
By definition
$
b_s(u,u)=\|u\|^2_{\cH^{\log,s}(\Omega)}.
$
It follows using the Lax-Milgram theorem that there exists a unique weak solution to problem \eqref{eq-omega}. Moreover, if $f\in L^{\infty}(\Omega)$ and $\Omega$ has Lipschitz boundary, then $u\in\cC(\Omega)$ (see \cite{KM17,M14, CS22}). If $\Omega$ further satisfies a uniform exterior sphere condition, the regularity results of \cite{KM14,CW19} apply to \eqref{eq-omega} and provide that
\[
u\in \cC_0(\R^N):= \big\{u\in\cC(\R^N):\quad u=0 \text{ on } \R^N\setminus\Omega\big\}.
\]
If additionally $f\in \cC^{\infty}(\Omega)$, then any weak solution $u$ of  problem \eqref{eq-omega} satisfies   (see \cite{FJ22}).
$$u\in \cC^{\infty}(\Omega).$$
Next,  consider the eigenvalue problem involving the operator $\logrels$, that is, we consider the problem \eqref{eq-omega} with $f=\lambda u$, $\lambda\in\R$,
\begin{equation}\label{eq-eigen}
	\begin{split}
	\quad\left\{\begin{aligned}
		(I+(-\Delta)^s)^{\log}  u &= \lambda u &&\text{ in\ \quad  $\Omega$}\\
		u &=  0             && \text{ on\ \  } \R^N\setminus\Omega. 
	\end{aligned}\right.
	\end{split}
	\end{equation}
Then there
is a sequence of eigenvalues
\[
0<\lambda_1(\Omega)<\lambda_2(\Omega)\le \cdots\le \lambda_k(\Omega)\quad\text{ with }\quad\lambda_k(\Omega)\to \infty~ \text{ as } ~ k\to \infty.
\]
The sequence $\{\phi_k\}_{k\in\N}$ of eigenfunctions corresponding to eigenvalues $\lambda_{k}(\Omega)$ forms a complete orthonormal basis of $L^2(\Omega)$ and an orthogonal system of $\cH_0^{\log,s}(\Omega)$.

We call  here a function $u\in \cH_0^{\log ,s}(\Omega)$ an eigenfunction  corresponding to the eigenvalue $\lambda$ if
\[
b_s(u,\phi)=\lambda\int_{\Omega}u\phi\ dx,\quad \text{ for all }\quad\phi\in \cH_0^{\log ,s}(\Omega).
\]
Moreover, it is straightforward to see  by integrating the inequality 
$$||u(x)|-|u(y)||\le |u(x)-u(y)|$$
that the bilinear form decreases  why taking the absolute values,
\[
b_s(|u|,|u|)\le b_s(u,u) \quad\text{ for }\quad u\in \cH_0^{\log,s}(\Omega).
\]
 Therefore, the first eigenvalue $\lambda_1(\Omega)$ is simple and the corresponding eigenfunction $\phi_1$
does not change its sign in $\Omega$ and can be chosen to be strictly positive in $\Omega$. The proof of the above facts can be deduced from \cite[Theorem 1.3]{F22}. 

In addition, any eigenfunction  $u\in \cH_0^{\log, s}(\Omega)$ of $\eqref{eq-eigen}$ is bounded, that is, $u\in L^{\infty}(\Omega)$ and there is a constant $C:=C(N,s)>0$ such that 
\[
\|u\|_{L^{\infty}(\Omega)}\le C\|u\|_{L^{2}(\Omega)}.
\]
This can be deduced from  \cite[Corollary 4.2]{FJ22} (see also \cite[Proposition 1.4]{F22}), using the $\delta$-decomposition technique as introduced in \cite{FJW22}

We note that since the operator $\logrels$ belongs to the family of nonlocal operators with small order, the strong and weak maximum principle  for  weak solutions involving the operator $\logrels$ can be deduced from \cite[Proposition 3.10 and Proposition 3.12]{FJ22} where general nonlocal operators of small order was considered see also \cite{JW19}. Since the first eigenvalue for $\logrels$ satisfies   $\lambda_1(\Omega)>0$ and the corresponding eigenfunction $\phi_1>0$ in $\Omega$, we provide the proof of the strong maximum principle for pointwise solutions here.
\begin{prop}[{\bf Strong Maximum Principle for pointwise solutions}]
    Let $\Omega\subset \R^N$ be an open bounded set. Let $c\in L^{\infty}(\Omega) $ be such that either $c\ge 0$ or $\lambda_1(\Omega)>c^-(x)$ in $\Omega$. Assume that  $u\in\cC(\overline \Omega)\cap \cL^s(\R^N)$ is  Dini  continuous in $\Omega$ and  satisfying pointwisely 
    \begin{equation}\label{eq-maxi}
	\begin{split}
	\quad\left\{\begin{aligned}
		(I+(-\Delta)^s)^{\log}  u &+ c(x) u\ge 0 &&\text{ in\ \quad  $\Omega$}\\
		u &\ge  0             && \text{ on\ \  } \R^N\setminus\Omega. 
	\end{aligned}\right.
	\end{split}
	\end{equation}
 Then $u\ge 0$ in $\R^N$. Furthermore,  either   $u>0$ in $\Omega$ or $u\equiv 0$ in $\R^N$.
\end{prop}
\begin{proof} To see that $u\ge 0$ in $\R^N$, we write  $u=u^+(x)-u^-(x)$, $x\in \R^N$ with 
$$\text{$ u^-(x)= -\min \{u(x),0\}$\quad and \quad$u^+=\max \{u(x),0\}$.}$$  Noticing that $(u^-(x)-u^-(x))(u^+(x)-u^+(y))\le 0$ in $\R^N\times\R^N$, we obverse that $u^-$ satisfies
\begin{equation}\label{eq-maxi1}
	\begin{split}
		(I+(-\Delta)^s)^{\log}  u^- \le  c^-(x) u^- \text{\quad  in\ \quad  $\Omega$}.
	\end{split}
	\end{equation}
Multiplying \eqref{eq-maxi1} by the first eigenfunction $\phi_1$  and integrate over $\R^N$ we obtain
\begin{align*}
     \lambda_1\int_{\R^N} u^-(x)\phi_1(x)\ dx&=\frac{1}{2}\iint_{\R^{2N}}(\phi_1(x)-\phi_1(y))(u^-(x)-u^-(y))K_s(x-y)\ dxdy\\
     &\le \int_{\R^N} c^-(x)u^-(x)\phi_1(x)\ dx.
\end{align*}
The calculations above yield  
 \begin{align*}   
\int_{\R^N} (\lambda_1(\Omega)-c^-(x))u^-(x)\phi_1(x)\ dx\le 0.
\end{align*}
Since $ \lambda_1(\Omega)-c^->0$ in $\Omega$,  this yields that $u^-\equiv 0$ in $\R^N$. Hence  $u$ is nonnegative in $\R^N$.   

Next, Suppose by contradiction that $u$ is not positive in $\Omega$. Since $\Omega$ is bounded, $\overline{\Omega}$ is compact. Since also $u$ is continuous in $\R^N$ and $u\ge 0$ in $\R^N\setminus\Omega$,  there is a point $x_0\in \Omega$ with
\begin{equation}\label{min}
u(x_0)=\min_{x\in\overline \Omega}u(x)\le 0.
\end{equation}
Therefore, as $x_0$ is an interior point where the minimum of $u$ is attained and $c(x_0)\ge 0$ in $\Omega$, it follows that 
\[
\logrels u(x_0) \ge -c(x_0)u(x_0)\ge 0.
\]
Whereas by \eqref{min}, we have that $u(x_0)\le u(x)$ for all $x\in \R^N$. It  follows that
\begin{align*}
0\le \logrels u(x_0)&=P.V.\int_{\R^N}(u(x_0)-u(y))K_s(x_0-y)\ dy\\
&=\int_{\R^N}(u(x_0)-u(y))K_s(x_0-y)\ dy\le 0.
\end{align*}
Moreover, since the integrand  is  non-positive by assumption and \eqref{min}, we conclude that 
\[
u\equiv u(x_0)\quad \text{ in }\quad\R^N.
\]
Now, since $u\ge 0$ in $\R^N\setminus\Omega$,  it follows that  $u\equiv 0$ in $\Omega$ and therefore  $u\equiv 0$ in $\R^N$. This leads to a contradiction and  the proof is established.
\end{proof}

%\newpage
\bibliographystyle{amsplain}

\end{document}